\documentclass[oneside,english]{amsart}

\usepackage[svgnames]{xcolor}

\usepackage{thmtools}
\usepackage{stmaryrd}
\SetSymbolFont{stmry}{bold}{U}{stmry}{m}{n}
\usepackage[english]{babel}
\usepackage[utf8]{inputenc}
\usepackage[T1]{fontenc}   
\usepackage{enumitem}
\usepackage{mathrsfs}
\usepackage{amssymb,amsmath,amsthm}
\usepackage{xstring}
\usepackage{mathtools}
\usepackage{tikz-cd}
\usepackage{csquotes}
\newcommand{\dd}[1]{\,\mathrm{d}#1}
\newcommand{\pdv}[3][]{\frac{\partial^{#1} #2}{\partial #3^{#1}}}
\newcommand{\pdvs}[3][]{\partial^{#1}#2/\partial^{#1}#3}
\usepackage{esint}

\newtheorem{theorem}{Theorem}[section]
\newtheorem{lemma}[theorem]{Lemma}

\theoremstyle{definition}
\newtheorem{definition}[theorem]{Definition}
\theoremstyle{plain}
\newtheorem{proposition}[theorem]{Proposition}
\newtheorem{corollary}[theorem]{Corollary}
\newtheorem*{theorem*}{Theorem}

\theoremstyle{remark}

\usepackage[
style=numeric,       
citestyle=numeric,   
sorting=none,       
doi=true,url=true,
backend=biber
]{biblatex}

\usepackage{accents}   
\newcommand{\D}{D}
\title[Regularity of viscosity solutions of the $\sigma_k$-Yamabe-Type Problem]{Regularity of viscosity solutions of the $\sigma_k$-Yamabe-type Problem for $k>n/2$}
\author{Wu Jinyang\, \orcidlink{0009-0002-6180-3619}}
\email{jwu124@uiowa.edu}
\address{Department of Mathematics,University of Iowa, Iowa City, IA 52240,	United States}
\date{July 3, 2024}

\usepackage[linktocpage]{hyperref}
\hypersetup{
	colorlinks=true,
	linkcolor=black,
	filecolor=black,
	citecolor = black,      
	urlcolor=black,
}
\usepackage{orcidlink} 
\begin{document}
\begin{abstract}
We study the regularity of Lipschitz viscosity solutions to the $\sigma_k$ Yamabe problem in the negative cone case. If either $k=n$ or the manifold is conformally flat and $k>n/2$, we prove that all Lipschitz viscosity solutions are smooth away from a closed set of measure zero. For the general $k>n/2$ case, under certain assumptions, we prove the existence of a Lipschitz viscosity solution that is smooth away from a closed set of measure zero. 
\end{abstract}

	\maketitle
\tableofcontents

\section{Introduction}
In this article, we study a fully nonlinear PDE in conformal geometry. First, we set up notations. 
Let $(M,g)$ be a smooth $n$-dimensional Riemannian manifold with $n \geq 3$. For $w \in C^\infty(M)$, let $g_w = e^{2w} g$ be a conformal metric in $[g]$. Let $A_{g_w}$ be the Schouten tensor of $(M,g_w)$:
\begin{equation*}
A_{g_w}= \frac{1}{n-2}\biggl(\text{Ric}_{g_w}-\frac{\text{Sc}_{g_w}}{2(n-1)}g_w \biggr),
\end{equation*}
where $\text{Ric}_{g_w}$ and $\text{Sc}_{g_w}$ denote the Ricci and scalar curvatures of $(M,g_w)$, respectively. Let $\lambda(g_w^{-1}A_{g_w})=(\lambda_1,\cdots, \lambda_n)$ be the set of eigenvalues of $g_w^{-1}A_{g_w}$. For $1 \leq k \leq n$, let $\sigma_k$ be the $k$-th elementary symmetric function 
\begin{equation*}
	\sigma_k(\lambda)= \sigma_k(\lambda_1,\cdots, \lambda_n)= \sum_{i_1 < \cdots < i_k}  \lambda_{i_1} \cdots \lambda_{i_k}.
\end{equation*}
Let $\Gamma^+_k$ be the positive $k$-cone
\begin{align*}
	\Gamma^{+}_k=\{\lambda=(\lambda_1,\cdots, \lambda_n); \sigma_1(\lambda) > 0, \cdots, \sigma_k(\lambda)>0\}.
\end{align*}

We consider the $\sigma_k$-Yamabe-type Problem in the negative cone case, namely 
\begin{equation}\label{sigmakproblem}\left\{
\begin{aligned}
	\sigma_k(\lambda(-g_w^{-1}A_{g_w}))&=R(x)>0,\\	\lambda(-g_w^{-1}A_{g_w}) & \in \Gamma_k^+,
\end{aligned}\right.
\end{equation}
where $R \in C^\infty(M)$, 
\begin{equation*}\left\{
	\begin{aligned}
	-g_w^{-1}A_{g_w}&= -g^{-1} \bigl(\frac{S(w,g)+A_g}{ e^{2 w}}\bigr),\\
		S(w,g) &= - \nabla^2_g w + d w \otimes d w -\frac{1}{2} |\nabla_g w|_g^2 g, 
	\end{aligned}\right.
\end{equation*}
and $\nabla_g$ is the covariant derivative with respect to the metric $g$. From the PDE's perspective, \eqref{sigmakproblem} is a fully nonlinear non-uniformly elliptic PDE. \eqref{sigmakproblem} is degenerate elliptic if we assume $\lambda(-g_w^{-1} A_{g_w}) \in \overline{\Gamma}_k^+$, instead of $\lambda(-g_w^{-1} A_{g_w}) \in \Gamma_k^+$.\par 

See  \autoref{sectionhistoricalresults} for some historical results of \eqref{sigmakproblem} and its variants. 

Viscosity solutions are a type of weak solution developed by {Crandall and Lions} in \cite{MR0690039}. See \cite{MR1351007} for some classical results on viscosity solutions of fully nonlinear uniformly elliptic PDEs. For viscosity solutions to \eqref{sigmakproblem}, we follow the definition of {Li, Nguyen, and Wang} in \cite{MR3813247}.

\begin{definition}\cite{MR3813247}\label{sigmakproblemviscositydefinitionelementary}
	$w \in C(M)$ is a viscosity subsolution to \eqref{sigmakproblem} if for all $x_0 \in M$, for all $\varphi \in C^2$, and for all neighborhoods $U$ of $x_0$ with 
	\begin{align*}
		\varphi(x_0)&= w(x_0) & \varphi \geq w \text{ in }U;
	\end{align*}
	there holds, 
	\begin{equation*}
		\begin{aligned}
			\lambda\bigl(-g_\varphi^{-1}A_{g_\varphi}\bigr)(x_0) &\in \Gamma_k^+ & &\text{ and } & \sigma_k\bigl[\lambda\bigl(-g_\varphi^{-1}A_{g_\varphi}\bigr)(x_0)\bigr] &\geq R(x_0).
		\end{aligned}
	\end{equation*} \par 
$w \in C(M)$ is a viscosity supersolution to \eqref{sigmakproblem} if for all $x_0 \in M$, for all $\varphi\in C^2$, and for all neighborhoods $U$ of $x_0$ with
	\begin{align*}
		\varphi(x_0)&= w(x_0) & \varphi \leq w \text{ in }U;
	\end{align*}
	either 
	\begin{equation*}
		\begin{aligned}
			\lambda\bigl(-g_\varphi^{-1}A_{g_\varphi}\bigr)(x_0) &\in \Gamma_k^+ & & \text{ and } &  \sigma_k\bigl[\lambda\bigl(-g_\varphi^{-1}A_{g_\varphi}\bigr)(x_0)\bigr] &\leq R(x_0)
		\end{aligned}
	\end{equation*}
or 
\begin{equation*}
	\lambda\bigl(-g_\varphi^{-1}A_{g_\varphi}\bigr)(x_0) \in \mathbb{R}^n-\Gamma_k^+.
\end{equation*}\par 
$w \in C(M)$ is a viscosity solution to \eqref{sigmakproblem} if it is both a viscosity subsolution and a viscosity supersolution to \eqref{sigmakproblem}.
\end{definition}

At the time of writing this paper, there are two approaches to construct viscosity solutions to \eqref{sigmakproblem}. The first approach is the Perron method. So far the Perron method works only on conformally flat manifolds. See \cites{MR3852835,MR4611402} for more details. The second approach is an elliptic regularization
method, where the viscosity solution is obtained as the limit of a sequence of smooth functions. See \cites{MR4210287,MR4956437} for more details. In \autoref{themaintheorem} and \autoref{corollary2}, we prove the same regularity result for viscosity solutions obtained via the Perron method and the elliptic regularization method, respectively, for $k > n/2$.

\subsection{Regularity of General Lipschitz Viscosity Solutions}
Our first main result is the following. 
\begin{theorem}\label{themaintheorem}
	Let $(M,g)$ be an $n$-dimensional Riemannian manifold with $n \geq 3$. Let $0 <R \in C^\infty(M)$ and $w$ be a locally Lipschitz viscosity solution to \eqref{sigmakproblem}. If either 
	\begin{enumerate}
	\item \label{themaintheorempt1} $(M,g)$ is
	conformally flat and $k > n/2$, or 
	\item \label{themaintheorempt2} $k=n$,
	\end{enumerate}
	 then there exists a closed zero measure set $S \subset M$ such that $w \in C^\infty (M\setminus S)$.
\end{theorem}
 \autoref{themaintheorem} is the first regularity result regarding viscosity solutions to \eqref{sigmakproblem}, and is a natural step in studying higher regularity of solutions following the existence results for viscosity solutions in \cites{MR3852835,MR4611402,MR4210287,MR4956437}.
In \autoref{themaintheorem}, we make no assumptions on the topology of $M$. The proof of \autoref{themaintheorem} is presented in \autoref{localregularitysection}, as a consequence of \autoref{maintheorempdeside}. We use {Chaudhuri and Trudinger}'s extension of the Alexandroff-Buselman-Feller Theorem \cite{MR2133413} to show that if either \eqref{themaintheorempt1} in \autoref{themaintheorem} or \eqref{themaintheorempt2} in \autoref{themaintheorem} is satisfied, then all Lipschitz viscosity solutions to \eqref{sigmakproblem} are punctually second-order differentiable in $M-S$, where $S$ is a measure zero subset. Next, we use the inverse function theorem to construct local smooth solutions to \eqref{sigmakproblem} in small neighborhoods of points in $M-S$. Finally, we show that about each point in $M-S$, there is a local smooth solution such that {Savin}'s regularity theory of small perturbation of viscosity solutions \cite{MR2334822} applies to a PDE defined by the local smooth solution. When $w$ is a viscosity solution to \eqref{sigmakproblem} on $(M,g)$ and neither \eqref{themaintheorempt1} in \autoref{themaintheorem} nor \eqref{themaintheorempt2} in \autoref{themaintheorem} is satisfied, our result \autoref{C2alpha} shows that the singular set \begin{equation}\label{singularsetclosed}
S=	\bigl\{ x_0 \in M;  w \text{ is not punctually second-order differentiable at } x_0 \bigr\}
\end{equation}
is closed. However, we have no further control on the size of the singular set $S$. 

It is well known that $\sigma_k^{1/k}$ is concave on $\Gamma_k^+$, see \cites{MR113978,MR1976082} for a reference. Our argument is independent of the concavity of $\sigma_k^{1/k}$ and applies to a wider class of PDE, see \autoref{C2alpha}. \par
 In \cite{MR4210287}, {Li and Nguyen} compute some explicit examples of viscosity solutions to \eqref{sigmakproblem} on annular domains $\{a < |x| < b\} \subset\mathbb{R}^n$ with the canonical Euclidean metric for $R=1$ and $2 \leq k \leq n$. These viscosity solutions are radial, singular on the level set $\{|x|=\sqrt{ab}\}$, and smooth elsewhere. These examples show that our result, \autoref{themaintheorem}, is optimal on Euclidean domains. \par  
By \cite[Theorem 3.5]{MR3852835}, if $(M,g)$ is conformally flat with $n \geq 3$ and $R=1$, then all viscosity solutions $w$ to \eqref{sigmakproblem} are locally Lipschitz. Hence, we deduce \autoref{corollary1} from \autoref{themaintheorem}. 
\begin{corollary}\label{corollary1}
	Let $(M,g)$ be a conformally flat manifold with $n \geq 3$. Let $R=1$, and $k > n/2$. If $w$ is a viscosity solution to \eqref{sigmakproblem}, then there exists a closed zero measure set $S \subset M$ such that $w \in C^\infty (M\setminus S)$.
\end{corollary}
Let $(M,g)$ be a Riemannian manifold with boundary $\partial M$. The boundary value problem,
\begin{equation}\left\{
	\begin{aligned}
		\sigma_k(\lambda(-g_w^{-1}A_{g_w}))&=R(x)>0,\\	\lambda(-g_w^{-1}A_{g_w}) & \in \Gamma_k^+,\\
		w(x) &\to \infty \qquad \text{as}\qquad d(x,\partial M) \to 0, \label{sigmakboundarytoinfty}
	\end{aligned}\right.
\end{equation}
where $d(\cdot,\partial M)$ is the distance to the boundary with respect to the metric $g$, is known as the $\sigma_k$-Loewner–Nirenberg problem. \par 
By \autocites[][]{MR3852835}[][Theorem 1.3]{MR4611402}, on a bounded smooth domain, $\Omega \subset \mathbb{R}^n$ with the canonical Euclidean metric, \eqref{sigmakboundarytoinfty} has a unique viscosity solution that is smooth near the boundary $\partial \Omega$, hence we deduce \autoref{corollary1.5} from \autoref{corollary1}.
\begin{corollary}\label{corollary1.5}
Let $\Omega \subset \mathbb{R}^n$ be a bounded smooth domain equipped with the canonical Euclidean metric. If $k > n/2$, $n\geq 3$, and $R=1$, then there exists a closed zero measure set $S \subset \Omega$ such that the unique viscosity solution $w$ to \eqref{sigmakboundarytoinfty} satisfies $w \in C^\infty (\Omega\setminus S)$.
\end{corollary}
Note that the viscosity solution in \cite[Theorem 1.3]{MR4611402} is constructed using the Perron method. \par 
\subsection{Regularity of Viscosity Solutions via the Elliptic Regularization Method}
In \cite[Theorem 1.3]{MR4210287}, {Li and Nguyen} show that for any closed Riemannian manifold $(M,g)$ with $\lambda(-g^{-1}A_g) \in \Gamma_k^+$ and $R=1$, there exists a Lipschitz viscosity solution $w$ to \eqref{sigmakproblem}. The proof in \cite{MR4210287} in fact works for all $0 < R \in C^\infty(M)$. The Lipschitz viscosity solution $w$ in  \cite[Theorem 1.3]{MR4210287} is the limit of a sequence of smooth functions $w_\tau$ as $\tau\to 1$. For $0 \leq  \tau <1$, $w_\tau$ is a smooth solution to the following modified $\sigma_k$-Yamabe-type problem :
\begin{align*}
A_{g_{w_\tau}}^{\tau} =\frac{1}{n-2}\left(\text{Ric}_{g_{w_\tau}}-\tau\cdot\frac{\text{Sc}_{g_{w_\tau}}}{2(n-1)}g_{w_\tau}\right),\\
\end{align*}
and
\begin{equation}\label{modifiedschouten}
\left\{ \begin{aligned}\sigma_{k}(\lambda(-g_{w_{\tau}}^{-1}A_{g_{w_{\tau}}}^{\tau})) & =R(x)>0,\\
	\lambda(-g_{w_{\tau}}^{-1}A_{g_{w_{\tau}}}^{\tau}) & \in\Gamma_{k}^{+}.
\end{aligned}
\right.
\end{equation}
Moreover, $w_\tau$ has uniform $C^0$ and $C^1$ estimates independent of $\tau$. By \autoref{punctuallydifferentiablealmosteverysmoothapproximation}, the viscosity solution $w$ is punctually second-order differentiable almost everywhere when $k > n/2$. Hence, with \autoref{C2alpha}, which shows the singular set \eqref{singularsetclosed} is closed, we deduce our second main result, \autoref{corollary2}. The proof of \autoref{corollary2} is presented in \autoref{localregularitysection}.
\begin{theorem}\label{corollary2}
Let $(M,g)$ be an $n$-dimensional closed Riemannian manifold with $n\geq 3$ and $\lambda(-g^{-1}A_g) \in \Gamma_k^+$. Let $0 < R \in C^\infty(M)$ and $k >n/2$. Then there exists a Lipschitz viscosity solution $w$ to \eqref{sigmakproblem} and a closed zero measure set $S \subset M$ such that $w \in C^\infty(M\setminus S)$. 
\end{theorem}
We also consider the following variant of the $\sigma_k$-Yamabe-type problem in the negative cone case. 
Let $(M,g)$ be an $n$-dimensional Riemannian manifold. We consider
\begin{equation} \label{Krylovstandard}\left\{
	\begin{aligned}
		\sigma_k\bigl(\lambda(-g_w^{-1}A_{g_w})\bigr)+\alpha(x) \sigma_{k-1}\bigl(\lambda(-g_w^{-1}A_{g_w})\bigr)&=\sum_{l=0}^{k-2}\alpha_l(x)\sigma_l\bigl(\lambda(-g_w^{-1}A_{g_w})\bigr),\\
		\lambda(-g_w^{-1}A_{g_w}) &\in \Gamma_{k-1}^+,\\
		0 < \alpha_l \text{ and } \alpha, \alpha_l \in C^\infty(M) \quad &\text{ for } 0 \leq l \leq k-2,\\ 
	\end{aligned}\right.
\end{equation}
where $\sigma_0(\lambda)=1$ for all $\lambda \in \mathbb{R}^n$. 
\eqref{Krylovstandard} is commonly known as the $\sigma_k$-Yamabe problem of Krylov type. \eqref{Krylovstandard} was first studied in \cites{MR4767573, MR4746060}, in which {Chen and He} prove global $C^0$ and $C^1$ estimates of classical solutions to \eqref{Krylovstandard}. Also see \cites{MR4387174,MR4278951,MR1284912} for some problems similar to \eqref{Krylovstandard}. \par 
Just like \eqref{sigmakproblem}, \eqref{Krylovstandard} is a non-uniformly elliptic fully nonlinear PDE. We can define viscosity solutions to \eqref{Krylovstandard} by replacing $\Gamma_k^+$ by $\Gamma_{k-1}^+$, $R(x_0)$ by $-\alpha(x_0)$, and $\sigma_k(\cdot )$ by 
\begin{equation*}
	\frac{\sigma_k}{\sigma_{k-1}}(\cdot)-\sum_{l=0}^{k-2} \alpha_l(x_0) \frac{\sigma_{l}}{\sigma_{k-1}}(\cdot),
\end{equation*}
in \autoref{sigmakproblemviscositydefinitionelementary}.\par 

Note that, unlike \eqref{sigmakproblem}, \eqref{Krylovstandard} is not homogeneous with respect to the eigenvalues of $-g^{-1}_wA_{g_w}$.
\begin{theorem}\label{corollary3}
Let $(M,g)$ be an $n$-dimensional closed Riemannian manifold with $n\geq 3$ and $\lambda(-g^{-1}A_g) \in \Gamma_k^+$. Let $0 < \alpha_l \in C^\infty(M)$ for $0 \leq l \leq k-2$, and $\alpha \in C^\infty(M)$. If $k-1>n/2$, then there exists a Lipschitz viscosity solution $w$ to \eqref{Krylovstandard} and a closed zero measure set $S \subset M$ such that $w \in C^\infty(M\setminus S)$. 
\end{theorem}
In \autoref{corollary3}, there is no assumption on the function $\alpha$ other than its smoothness. 
The existence of the Lipschitz viscosity solution $w$ to \eqref{Krylovstandard} comes from \cite{MR4767573}. The viscosity solution $w$ is also the limit of a sequence of uniformly bounded and uniformly Lipschitz solutions to the modified $\sigma_k$-Yamabe problem of Krylov type. See \autoref{Krylovexistence} for more details. In particular, \autoref{Krylovexistence} also shows that the viscosity solution $w$ is punctually second-order differentiable almost everywhere when $k-1 > n/2$. The closedness of the singular set \eqref{singularsetclosed} of the viscosity solution $w$ in \autoref{corollary3} follows from \autoref{C2alpha} with some modification in the proof. See \autoref{Krylovtypesection} for more details.\par 

At the time of writing this paper, there are no results on the existence of classical solutions to equations \eqref{sigmakproblem}, \eqref{sigmakboundarytoinfty}, or \eqref{Krylovstandard} that are known to the author. In fact, {Li, Nguyen, and Xiong} \cite{MR4611402} show the nonexistence of classical solutions to \eqref{sigmakboundarytoinfty} when $k \geq 2$ and $(\Omega,g)$ is a bounded Euclidean domain with more than one boundary component. Our results, specifically \autoref{corollary1.5}, \autoref{corollary2}, and \autoref{corollary3}, further justify the consideration of viscosity solutions to each of these problems in the corresponding cases, building on the work of {Li, Nguyen, and Xiong} \cite{MR4611402}.
\subsection{Future Directions}
For future directions, when $k \leq n/2$, we do not know any examples of Lipschitz viscosity solutions to \eqref{sigmakproblem} with singular sets of positive measure. If such examples exist, a follow-up question would be: What are the sufficient conditions on the viscosity solutions so that their singular sets are of small measure? 

Another question is related to the base metrics of the Riemannian manifolds. Note that in \autoref{corollary1.5}, \autoref{corollary2}, and \autoref{corollary3}, we have assumptions on the base metric in each of them to ensure the existence of viscosity solutions. Such kind of assumptions are common in the field. However, \autoref{nonkadmissible} shows that such kind of assumptions are not necessary when $k$ is small. Also see \cite{MR4817303}, where {Yuan} proves, for compact manifolds with boundary and $2 \leq k <n/2$, the existence of a smooth function $w$ such that $\lambda(-g_w^{-1}A_{g_w}) \in \Gamma_k^+$.

\begin{theorem}\cite[Theorem 1.1]{MR4956437}\label{nonkadmissible}
	Let $n\geq 3$, $2\leq  k < n/2$, $R=1$, and $(M,g)$ be a compact $n$-dimensional Riemannian manifold with nonempty boundary $\partial M$. Then the $\sigma_k$-Loewner–Nirenberg problem \eqref{sigmakboundarytoinfty} has a locally Lipschitz viscosity solution $w$. 
\end{theorem}
Our result \autoref{themaintheorem} does not apply to the viscosity solution $w$ in \autoref{nonkadmissible}. Note that when $(M,g)$ is a bounded smooth Euclidean domain, by \cite{MR4611402}, the unique viscosity solution is smooth near the boundary of the Euclidean domain. Given \autoref{nonkadmissible} and \cite{MR4611402}, a future question would be: For \autoref{nonkadmissible}, if we assume $\lambda(-g^{-1}A_g)\in \Gamma_k^+$, could the viscosity solution in the theorem be punctually second-order differentiable almost everywhere?

We have one more question related to the sizes and the shapes of the singular sets. While in \autoref{corollary2}, the singular sets of the Lipschitz viscosity solutions to \eqref{sigmakproblem} are small when $k>n/2$, it is unclear if these viscosity solutions can have point singularities, especially when $R=1$. For comparison, we note the following result. In \cite{MR4880431}, {Shankar and Yuan} study viscosity solutions to the following constant Hessian equations on Euclidean domains,
\begin{align}\label{constanthessian}
	\sigma_k (\nabla^2 w)&=1, & \lambda(\nabla^2 w)& \in \Gamma_k^+. 
\end{align}
Using an a priori estimate only available for \eqref{constanthessian}, {Shankar and Yuan} prove an extension of the Alexandroff-Buselman-Feller Theorem, and show the singular sets of viscosity solutions to \eqref{constanthessian} are of measure zero for arbitrary $n,k$. Additionally, when $n=3,4$, and $k=2$, the singular sets are removable, hence viscosity solutions are smooth \autocites[][]{MR2487850}[][Remark 5.1]{MR4880431}. We wonder when $R=1$ and $k>n/2$, if the singular sets of the viscosity solutions in \autoref{corollary2} can have Hausdorff dimension less than $n-1$.  \par 

Our paper is structured as follows. In \autoref{sectionhistoricalresults}, we state some historical results related to \eqref{sigmakproblem}. In \autoref{seclocalsolution}, we set up the notations and state some elementary properties of viscosity solutions. In \autoref{thenewsection}, we prove a modified version of {Chaudhuri and Trudinger}'s extension of the Alexandroff-Buselman-Feller Theorem. We also show that such modified version applies to the viscosity solutions in \autoref{corollary2} and \autoref{corollary3}. In \autoref{localregularitysection}, we prove \autoref{themaintheorem}. In \autoref{Krylovtypesection}, we extend the results in earlier sections to the $\sigma_k$-Yamabe problem of Krylov type.

\subsection*{Acknowledgments}
The author would like to thank Hao Fang for introducing this problem and for providing guidance and suggestions. Additionally, the author expresses gratitude to Lihe Wang for helpful discussions and guidance, particularly regarding the use of the inverse function theorem to construct local solutions. The author would also like to thank Biao Ma and Wei Wei for helpful discussions. The author would also like to thank Jonah Duncan, He Yan, Luc Nguyen, Li Chen, and the two reviewers for their comments on the paper.\par 
The author also acknowledges partial support from the Erwin and Peggy Kleinfeld Graduate Fellowship fund.

\section{Overview of Known Results}
\label{sectionhistoricalresults}
	\subsection{$\sigma_k$-Yamabe-type Problem In The Positive Cone Case}
	
	\eqref{sigmakproblem} has been studied extensively under a different cone condition, that is,
	\begin{equation}\label{sigmakproblempositive}\left\{
		\begin{aligned}
			\sigma_k(\lambda(g_w^{-1}A_{g_w}))&=R(x)>0,\\	\lambda(g_w^{-1}A_{g_w}) & \in \Gamma_k^+.
		\end{aligned}\right.
	\end{equation}
	\eqref{sigmakproblempositive} is mostly known as the $\sigma_k$-Yamabe-type Problem ($k\geq 2$) in the positive cone case.
	Here we state some historical results related to \eqref{sigmakproblempositive}. We do not intend to be comprehensive here.\par 
	\begin{definition}\cite{MR1738176}
		Let $(M,g)$ be a Riemannian manifold. The metric $g$ is $k$-admissible if $\lambda(g^{-1}A_g) \in \Gamma_k^+$.
	\end{definition}
	Problem \eqref{sigmakproblempositive} was initiated by {Viaclovsky} in \cite{MR1738176}. In \cite{MR1925503}, {Viaclovsky} studies \eqref{sigmakproblempositive} on closed $k$-admissible Riemannian manifolds $(M,g)$ and proves global $C^2$ and $C^1$ estimates for classical solutions to \eqref{sigmakproblempositive} assuming global $C^0$ estimates. Concurrently, with a completely different approach, {Chang, Gursky, and Yang} prove in \cite{MR1945280} that for closed $2$-admissible $4$-dimensional manifolds $(M,g)$ that are not conformally equivalent to the round sphere, classical solutions to \eqref{sigmakproblempositive} have global $C^0$ and $C^1$ estimates when $k=2$. Later in \cite{MR1976045}, using the concavity of $\sigma_k^{1/k}$, {Guan and Wang} prove local $C^2$ and $C^1$ estimates for classical solutions to \eqref{sigmakproblempositive} on Riemannian manifolds assuming local $C^0$ estimates. See also \cite{MR2204639} for a different approach and \cite{MR1988895} for some other related works in this direction. Recently in \cite{chu2023liouville}, {Chu, Li, and Li} give a proof of local $C^1$ estimates from $C^0$ estimates of classical solutions to \eqref{sigmakproblempositive}  without using the concavity of $\sigma_k^{1/k}$.  In case \eqref{sigmakproblempositive} is degenerate, i.e.,  assuming a non-strict inequality $R \geq 0$, {Trudinger and Wang} \cite{MR2481828} prove a Harnack inequality independent of $R$ when $k > n/2$ for closed  $k$-admissible Riemannian manifolds not conformally equivalent to the round sphere. \par
	Current existence results for classical solutions to \eqref{sigmakproblempositive} as well as the $C^0$ estimates depend heavily on the conformal structure of the closed Riemannian manifolds and the size of $k$. For each result listed in this paragraph, the set of classical solutions to \eqref{sigmakproblempositive} is compact in the $C^m$ topology for $m \geq 0$, assuming additionally that the corresponding manifolds are not conformally equivalent to the round sphere. {Gursky and Viaclovsky} \cite{MR2078344} define $k$-maximal volume for closed Riemannian manifolds when $k \geq n/2$ and give a sufficient condition for the existence of smooth solutions to \eqref{sigmakproblempositive}. In particular, let $(M,g)$ be a $4$-dimensional closed $k$-admissible manifold,  if $R = 1$, then \eqref{sigmakproblempositive} has smooth solutions for $2 \leq k \leq 4$. In a later work \cite{MR2373147}, {Gursky and Viaclovsky} prove the existence of smooth solutions to \eqref{sigmakproblempositive} provided that $(M,g)$ is closed and $k$-admissible with $k > n/2$. \par 
	
	When $R=1$, if either $k=2$ or $(M,g)$ is conformally flat, then \eqref{sigmakproblempositive} is variational. {Sheng, Trudinger, and Wang} \cite{MR2362323} show that when \eqref{sigmakproblempositive} is variational, $2\leq k \leq n/2$, and $(M,g)$ is a closed Riemannian manifold with $k$-admissible metric, then \eqref{sigmakproblempositive} admits a smooth solution. Their work is an extension of the result in \cite{MR1945280}. They show the existence of classical solutions to \eqref{sigmakproblempositive} when $k=2$, $n=4$, and the closed manifolds have $2$-admissible metrics. Concurrent with \cite{MR2362323}, {Ge and Wang} \cite{MR2290138} prove independently the existence of classical solutions to \eqref{sigmakproblempositive} when $k = 2$ and $R = 1$ on compact orientable Riemannian manifolds of dimension $>8$ with $2$-admissible and non-conformally flat metric. See also \cites{MR1988895,MR2233687}, where {Li and Li} prove the existence of solutions to \eqref{sigmakproblempositive} when $R= 1$ on conformally flat $k$-admissible manifolds. Additionally, they prove the compactness of the solution set in $C^3$ topology when the manifold is not conformally equivalent to the round sphere. Their work is an extension to \cite{MR1978409} that only covers the cases when $k \neq n/2$. 
	
	In \cite{MR3165241}, {Li and Nguyen} consider closed Riemannian manifolds $(M,g)$ that are not conformally equivalent to the round sphere. For $R=1$, {Li and Nguyen} prove the set of solutions $w$ to \eqref{sigmakproblempositive}, satisfying certain lower bounds on the Ricci curvature of $g_w$, is compact in the $C^5$ topology. In case when $k \geq n/2$, the lower Ricci bound assumption is automatically satisfied. Additionally, when $R = 1$ and $k \geq n/2$, {Li and Nguyen} prove the existence of solutions to \eqref{sigmakproblempositive} for all compact $k$-admissible Riemannian manifolds.  \par 
	Most of the results above require the manifold to be $k$-admissible. See \cites{MR1923964,MR2629509,MR4632573,MR4817303} for some results on the existence of $k$-admissible metrics. \par 
	When $(M,g)$ is the canonical Euclidean space $\mathbb{R}^n$, solutions to \eqref{sigmakproblempositive} are commonly referred to as entire solutions. In \cite{MR2165306}, {Chang, Han, and Yang} give a classification of radial entire solutions to both \eqref{sigmakproblem} and \eqref{sigmakproblempositive} when $R=1$. For rotational symmetry of entire solutions to \eqref{sigmakproblempositive}, see \cites{MR1694380,MR2055838,MR1988895,MR2233687} for some historical results. In particular, entire solutions to $\lambda(g_w^{-1}A_{g_w}) \in \partial \Gamma_k^+$ are constants, see \cites{MR2547976,MR3813247} for some references. For some recent results in this direction, in \cite{MR4584983}, {Fang, Ma, and Wei} give a proof of rotational symmetry of entire solutions to \eqref{sigmakproblempositive} on $\mathbb{R}^4$ when $k = 2$ and $R=1$, whose proof does not use the Kelvin transform and also applies to other non-geometric PDEs of similar type. \par 
	
	When $(M,g)$ is the round sphere, problem \eqref{sigmakproblempositive} is known as the $\sigma_k$-Nirenberg problem. There is a vast amount of literature on the Nirenberg problem. We refer the reader to the following papers  
	\cites{MR4691488,MR4458997, MR1680925, MR0925123, MR0908146} and the references therein for some insights into this subject. Also see \cites{MR3964274, MR4241796} for \eqref{sigmakproblempositive} on spheres with conic singularities.
	For the Dirichlet problems associated with \eqref{sigmakproblempositive} and the related boundary a priori estimates, see \cites{MR2333094,MR2570314,MR2293983,MR3130336} for an incomplete list of references. For the uniqueness of solutions to \eqref{sigmakproblempositive}, see \cites{MR3759362,MR3959563} for an incomplete list of references.\par

\subsection{$\sigma_k$-Yamabe-type Problem In The Negative Cone Case}
\eqref{sigmakproblem} was initially proposed in \cite{MR1976082}. In \cite{MR1976082}, {Gursky and Viaclovsky} prove global $C^0$ and $C^1$ estimates of classical solutions to \eqref{sigmakproblem} on closed Riemannian manifolds.  In \cite{chu2023liouville}, {Chu, Li, and Li} prove a local $C^1$ a priori estimate for classical solutions to \eqref{sigmakproblem} assuming a one-sided $C^0$ estimate when $2k \neq n$. However, unlike the positive cone case \eqref{sigmakproblempositive}, there are no known $C^2$ a priori estimates for classical solutions to \eqref{sigmakproblem}, see \cite{MR2362323} for some related counterexamples. For geometric applications of \eqref{sigmakproblem}, see \cite{MR4401830}.
Also see \cite{chu2023liouville}, in which {Chu, Li, and Li} use the Kelvin transform to prove that all entire solutions to $ \lambda(-g_w^{-1}A_{g_w}) \in \partial \Gamma_k^+$ are constant when $k \neq n/2$. They also prove the radial symmetry of viscosity solutions to some related conformally invariant fully nonlinear equations on $\mathbb{R}^n-\{0\}$.
 \par 

Unlike the $\sigma_k$-Yamabe problem in the positive cone case, there are no known results regarding the existence of classical solutions for either \eqref{sigmakproblem} or \eqref{sigmakboundarytoinfty}. In \cites{MR3813247,MR3852835}, {Gonz{\'a}lez, Li, Nguyen, and Wang} initiate the study of viscosity solutions to \eqref{sigmakproblem} by proving a comparison principle for viscosity subsolutions and viscosity supersolutions to \eqref{sigmakproblem} on bounded Euclidean domains. We summarize the main results of {Gonz{\'a}lez, Li, Nguyen, and Xiong} from the sequence of works \cites{MR3852835,MR4210287, MR4611402} as follows.
\begin{theorem}\autocites[][Theorem 1.3]{MR4210287}[][Theorem 1.3]{MR4611402}
	
	\label{li2021ath}
	Let $n\geq 3$, $2 \leq k \leq n$. Then 
	\begin{enumerate}
		\item if $(M,g)$ is a closed Riemannian manifold with $\lambda(-g^{-1}A_g) \in \Gamma_k^+$, then \eqref{sigmakproblem} has a Lipschitz viscosity solution for positive functions $R \in C^\infty(M)$;\label{li2021ath1}
		\item \label{li2021ath2}if $(\Omega,g)$ is a smooth bounded Euclidean domain with the canonical Euclidean metric, then the $\sigma_k$-Loewner–Nirenberg problem \eqref{sigmakboundarytoinfty} has a unique viscosity solution $w$ when $R=1$. Moreover, $w$ is locally Lipschitz and smooth in a neighborhood of $\partial \Omega$; \item if $(\Omega, g)$ is a smooth bounded Euclidean domain with the canonical Euclidean metric and $\partial \Omega$ has more than one connected component, then there is no classical solution to \eqref{sigmakboundarytoinfty} when $R=1$. \label{li2021ath3}
	\end{enumerate}
\end{theorem}
Given \autoref{li2021ath}\eqref{li2021ath3}, our result \autoref{themaintheorem} provides further motivation for studying viscosity solutions to \eqref{sigmakboundarytoinfty} on conformally flat manifolds when $k>n/2$ and $R=1$. \par 
In \cite{MR4956437}, {Duncan and Nguyen} remove the assumption on the metric of the manifold in \autoref{li2021ath}\eqref{li2021ath2} for $k < n/2$; see \autoref{nonkadmissible} for details. In a later work \cite{MR4978553}, {Duncan and Wang} study a problem similar to \eqref{sigmakproblem} on non-compact manifolds and prove some existence results.\par 
For an approach different from the viscosity solution, see \cite{MR4654029}, in which {Duncan} studies classical solutions $w$ to \eqref{sigmakproblem} and proves
\begin{equation*}
\Vert \D^2 w\Vert _{L^\infty(\Omega^\prime)} \leq C(\Vert \D^2 w\Vert _{L^p(\Omega)}, \Omega,\Omega^\prime,p,k,M,g),
\end{equation*}
for some $\Omega^\prime \Subset \Omega \subset M$ and $p$ depending on $n$ and $k$. Also see \cite{MR4293879}, in which {Duncan and Nguyen} prove local pointwise second derivative estimates for strong $W^{2,p}$ solutions to the $\sigma_k$-Yamabe problem in both the positive and negative cone cases. 

\section{Set up and Technical Tools}\label{seclocalsolution}
To study \eqref{sigmakproblem}, we start with a more general type of PDE. \par 
Let $B_1 = \{|x| <1; x \in \mathbb{R}^n\}$. Let $\text{Sym}_n$ be the space of $n$ by $n$ symmetric matrices, and $\text{Sym}_n^+$ be the space of positive definite symmetric matrices,
\begin{equation}
	\begin{aligned}
		\text{Sym}_n&=\{M=(M_{ij})_{n\times n}; M_{ij}=M_{ji}\},\\
			\text{Sym}_n^+&=\{M>0; M \in \text{Sym}_n\}.
	\end{aligned}
\end{equation}
Let $\lambda \colon \text{Sym}_n \to \mathbb{R}^n$ be a map that sends a symmetric matrix to its eigenvalues.
Let $(f, \Gamma)$ be a pair such that 
\begin{equation}\label{conecondition}
	\left\{
\begin{aligned}
	&\Gamma \subset \mathbb{R}^n\text{ is an open symmetric cone with vertex at the origin},\\
	& \Gamma \supset \Gamma+\Gamma_n^+=\{\lambda+\mu; \lambda\in \Gamma, \mu \in \Gamma_n^+\},\\ 
	& f \in C^{2,1}(\Gamma) \cap C^0(\overline{\Gamma}),\\
	&f \text{ is homogeneous of degree one},\\
	& f>0 \text{ in }\Gamma,\quad   \pdv{f }{\lambda_i}>0 \text{ in } \Gamma, \quad f=0 \text{ on }\partial \Gamma. 
\end{aligned}\right.
\end{equation}
See \cites{MR3813247,MR4611402,chu2023liouville} for similar setups.
By \autocites[][]{MR113978}[][Proposition 2.2]{MR1976082}
,
 the pair $(\sigma_k^{1/k}, \Gamma_k^+)$ satisfies the above conditions.\par 
In this section, we study the following type of PDE, 
\begin{equation}\label{moregeneralL}\left\{
\begin{aligned}
	f\bigl(\lambda\bigl( W(x) \circ (\D^2 w + L(x,\D w)) \circ W(x)  \bigr)\bigr) &= R(x)e^{2w}>0,\\
	\lambda( W(x) \circ (\D^2 w + L(x,\D w))\circ W(x)  ) & \in \Gamma,
\end{aligned}\right.
\end{equation}
where $L\in C^{2,1}(\overline{B}_1 \times \mathbb{R}^n,\text{Sym}_n)$, $W\in C^{2,1}(\overline{B}_1,\text{Sym}_n^+)$. We use $\D$ to denote the derivative taken under the flat metric. We use $\circ$ to denote matrix composition. In addition, we assume the following structure condition on $L$.
\begin{equation}\label{structurealconditionsoinlnL}\left\{
\begin{aligned}
	L(x,p)&=L_0(x,p)+ L_1(x,p)+ L_2(x,p), \\
	L_k(x,rp)&= r^k L_k(x,p), \quad k=0,1,2.
\end{aligned}\right.
\end{equation}\par 
The $\sigma_k$-equation \eqref{sigmakproblem} is a special case of \eqref{moregeneralL}. Fix local coordinates $x=(x^1,\cdots, x^n)$, let $g_{ij}=g(\partial x^i, \partial x^j)$, we take $W_{ij}(x)=\sqrt{g^{ij}(x)}$, and
\begin{equation}\label{Lforschouten}\left\{
\begin{aligned}
	L_{0,ij}(x,p) &= -(A_g)_{ij}(x), \\
	L_{1,ij}(x,p) &= -\Gamma_{ij}^k(x) p_k, \\
	 L_{2,ij}(x,p)&=  - p_i \otimes p_j + \frac{1}{2} (g^{kl}(x) p_kp_l) g_{ij}(x),
\end{aligned}\right.
\end{equation}
where $A_g$ is the Schouten Curvature tensor of $(B_1,g)$, $\Gamma_{ij}^k$ are the Christoffel symbols induced by $g$, and $(g^{kl})$ is the inverse matrix of $(g_{ij})$, so that $g^{il}g_{lj}=\delta_{ij}$.
 \par 
For simplicity, we denote  
\begin{equation*}
L(\cdot, \D w)(x_0)= L(x_0,\D w(x_0)),
\end{equation*}
for $w \in C^1(\overline{B}_1)$ and $x_0 \in \overline{B}_1$. 

\begin{definition}\cite{MR3813247}\label{defnviscositysolution}
	$w \in C(B_1)$ is a viscosity subsolution to \eqref{moregeneralL} if for all $x_0 \in B_1$, for all $\varphi \in C^2$ and for all neighborhoods $U$ of $x_0$ with 
	\begin{align}
		\varphi(x_0)&= w(x_0), & \varphi \geq w \text{ in }U,
	\end{align}
	there holds, 
	\begin{equation*}
		\begin{aligned}
			f\bigl[\lambda\bigl(W \circ (\D^2 \varphi + L(\cdot ,\D \varphi))\circ W  \bigr)\bigr](x_0) \geq  R(x_0) e^{2\varphi(x_0)},& \text{ and }\\
		\lambda\bigl( W \circ (\D^2 \varphi + L(\cdot ,\D \varphi))  \circ W)(x_0) \in \Gamma. &
		\end{aligned}
	\end{equation*} \par 
	$w \in C(B_1)$ is a viscosity supersolution to \eqref{moregeneralL} if for all $x_0 \in B_1$, for all $\varphi\in C^2$ and for all neighborhoods $U$ of $x_0$ with
	\begin{align}
		\varphi(x_0)&= w(x_0) & \varphi \leq w \text{ in }U,
	\end{align}
	either 
\begin{equation*}
	\begin{aligned}
		f\bigl[\lambda\bigl( W \circ (\D^2 \varphi + L(\cdot ,\D \varphi)) \circ W \bigr)\bigr](x_0) \leq  R(x_0) e^{2\varphi(x_0)},&\text{ and }\\ 
		\lambda\bigl( W \circ (\D^2 \varphi + L(\cdot ,\D \varphi)) \circ W \bigr)(x_0) \in \Gamma,&
	\end{aligned}
\end{equation*}
	or 
	\begin{equation*}
	\lambda\bigl( W \circ (\D^2 \varphi + L(\cdot ,\D \varphi)) \circ W  \bigr)(x_0) \in \mathbb{R}^n \setminus \Gamma.
	\end{equation*}\par 
	$w \in C(B_1)$ is a viscosity solution to \eqref{moregeneralL} if it is both a viscosity subsolution and a viscosity supersolution to \eqref{moregeneralL}.
\end{definition}
\begin{definition}\cite{MR1351007}
A paraboloid is a polynomial of degree $2$.
$w$ is punctually second-order differentiable at $x_0$ if there exists a paraboloid $P$ such that 
\begin{equation}
	w(x)=P(x)+o(|x-x_0|^2),
\end{equation}
as $x \to x_0$. In this case, we define $\D w(x_0)=\D P(x_0)$ and $\D^2 w(x_0)=\D^2 P(x_0)$.
\end{definition}
 
Our work is based on the following observation. 
\begin{lemma}\label{punctuallysecondincone}
Let $w$ be a viscosity solution to \eqref{moregeneralL}. If $w$ is punctually second-order differentiable at $x_0$, then
\begin{equation*}
	\lambda\bigl( W \circ (\D^2 w + L(\cdot ,\D w) )\circ W \bigr)(x_0) \in \Gamma.
\end{equation*}  Hence,
\begin{equation*}
f\bigl[\lambda\bigl( W \circ (\D^2 w + L(\cdot ,\D w))  \circ W \bigr)\bigr](x_0) = R(x_0) e^{2w(x_0)}.
\end{equation*} 
\end{lemma}
\begin{proof}
	Let $P$ be the paraboloid such that 
	\begin{equation}\label{punctuallysecondorderwde}
		w(x)=P(x)+o(|x-x_0|^2),
	\end{equation}
as $x \to x_0$.
Then by \eqref{punctuallysecondorderwde}, for all $\epsilon>0$ there exists a neighborhood $U$ of $x_0$ such that
\begin{align}
 P(x_0)&= w(x_0) &  P_{\epsilon}(x)\coloneqq P(x)+ \epsilon |x-x_0|^2 \geq w  \text{ in }U. \label{punctuallysecondorderpushup}
\end{align}
Hence by \autoref{defnviscositysolution} and \eqref{punctuallysecondorderpushup},
\begin{equation}\label{punctuallysecondorderpushupse}
\left\{\begin{aligned}
	f\bigl(\lambda\bigl(W \circ \bigl(\D^2 P_\epsilon+ L(\cdot,\D P_\epsilon )\bigr) \circ W\bigr)\bigr)(x_0) &\geq R(x_0) e^{2w(x_0)},\\
\lambda\bigl(W \circ \bigl(\D^2 P_\epsilon+ L(\cdot,\D P_\epsilon) \bigr) \circ W\bigr)(x_0) &\in \Gamma.
\end{aligned}\right.
\end{equation}
Pushing $\epsilon \to 0$, we conclude from \eqref{punctuallysecondorderpushupse} that 
\begin{equation}\label{punctuallysecondorderpushupcon}
	\left\{\begin{aligned}
		f\bigl(\lambda\bigl(W \circ \bigl(\D^2 P+ L(\cdot,\D P )\bigr) \circ W\bigr)\bigr)(x_0) &\geq R(x_0) e^{2w(x_0)},\\
		\lambda\bigl(W \circ \bigl(\D^2 P+ L(\cdot,\D P)\bigr) \circ W\bigr)(x_0) &\in \overline{\Gamma}.
	\end{aligned}\right.
\end{equation}
Since $R(x_0)>0$ and $f=0$ on $\partial \Gamma$ by \eqref{conecondition}, we conclude that 
\begin{equation}\label{liestrictlyincone}
\lambda\bigl(W \circ \bigl(\D^2 P+ L(\cdot,\D P)\bigr) \circ W\bigr)(x_0) \in \Gamma.
\end{equation}\par 
By \eqref{punctuallysecondorderwde}, for all $\epsilon>0$ there exists a neighborhood $U$ of $x_0$ such that
 \begin{align*}
 	P(x_0)&= w(x_0) &  P_{-\epsilon}(x)\coloneqq P(x)- \epsilon |x-x_0|^2 \leq w  \text{ in }U. 
 \end{align*}
By \eqref{liestrictlyincone}, for $\epsilon > 0$ small,
\begin{equation*}
\lambda\bigl(W \circ \bigl(\D^2 P_{-\epsilon}+ L(\cdot,\D P_{-\epsilon})\bigr) \circ W\bigr)(x_0) \in \Gamma.
\end{equation*}
Hence by \autoref{defnviscositysolution},
\begin{equation*}
f\bigl(\lambda\bigl(W \circ \bigl(\D^2 P_{-\epsilon}+ L(\cdot,\D P_{-\epsilon} )\bigr) \circ W\bigr)\bigr)(x_0) \leq R(x_0) e^{2w(x_0)}.
\end{equation*}
Pushing $\epsilon \to 0$, we get 
\begin{equation}\label{punctuallysecondorderpushupdown}
f\bigl(\lambda\bigl(W \circ \bigl(\D^2 P+ L(\cdot,\D P )\bigr) \circ W\bigr)\bigr)(x_0) \leq R(x_0) e^{2w(x_0)}.
\end{equation}
\eqref{punctuallysecondorderpushupcon} and \eqref{punctuallysecondorderpushupdown} together give us the equality 
\begin{equation*}
f\bigl(\lambda\bigl(W \circ \bigl(\D^2 P+ L(\cdot,\D P )\bigr) \circ W\bigr)\bigr)(x_0) = R(x_0) e^{2w(x_0)},
\end{equation*}
completing the proof.
\end{proof}
\autoref{punctuallysecondincone} indicates that \eqref{moregeneralL} is elliptic at points where the viscosity solution $w$ is punctually second-order differentiable. \par 
\autoref{Savinisgod} is our main technical tool. We use \autoref{Savinisgod} to show that if a viscosity solution $w$ to \eqref{sigmakproblem} is punctually second-order differentiable at a point, then $w$ is $C^{2,\alpha}$ in a neighborhood of that point. See \autoref{localregularitysection} for more.
\begin{theorem}\cite[Theorem 1.3]{MR2334822}\label{Savinisgod}
	Suppose $F\colon \text{Sym}_n \times \mathbb{R}^n \times \mathbb{R} \times B_1 \to \mathbb{R}$ satisfies the following hypothesis for $|p| \leq \delta$, $|z| \leq \delta$,
	\begin{enumerate}[label=H\arabic*]
		\item \label{SavinP0} $F(\cdot ,p,z,x)$ is elliptic, i.e., 
		\begin{equation*}
		F(M+N,p,z,x) \geq F(M,p,z,x)
		\end{equation*}
	if $N \geq 0$.
		\item \label{SavinP1} $F(M,p,z,x)$ is uniformly elliptic in a $\delta$-neighborhood of the origin with ellipticity constants $\Lambda \geq \lambda > 0$, i.e., 
		\begin{equation*}
			\Lambda \Vert N \Vert \geq F(M+N,p,z,x)-F(M,p,z,x) \geq \lambda \Vert N \Vert
		\end{equation*}
		when $N \geq 0$, $\Vert M\Vert , \Vert N \Vert \leq \delta$, $x \in B_{1}$.
		\item \label{SavinP2.5} $F(0,0,0,x)=0$.
		\item \label{SavinP2}$F \in C^2$ and $\Vert \D F \Vert ,\Vert \D^2 F\Vert  \leq K$ in a $\delta$-neighborhood of the set 
		\begin{equation*}
			\{(0,0,0,x); x \in B_{1}\}.
		\end{equation*}
	\end{enumerate}\par 
Let $u$ be a viscosity solution to 
\begin{equation*}
F(\D^2 u, \D u, u, x)=0.
\end{equation*}
Then there is a constant $c_1$ depending on $\delta$, $K$ and the universal constants $n,\lambda,\Lambda$ such that  
\begin{equation*}
	|u|_{0,B_1} \leq c_1,
\end{equation*}
implies $u \in C^{2,\alpha}(B_{1/2})$, and 
\begin{equation*}
|u|_{2,\alpha;B_{1/2}} \leq \delta. 
\end{equation*}
\end{theorem}

\autoref{kconvexalexdandrofftype} is an extension of the Alexandroff-Buselman-Feller theorem.
\begin{definition}\label{kconvexitydefn}\cite{MR2133413} 
A function $w \in C(B_1)$ is $k$-convex in the viscosity sense, if for any $x_0 \in B_1$ and any paraboloid $\varphi$ with
\begin{align}
	\varphi(x_0)&= w(x_0) & \varphi \geq w \text{ in a neighborhood of }x_0,
\end{align}
 we have 
	\begin{equation}\label{kconvexequation}
		\lambda(\D^2 \varphi)	(x_0) \in \overline{\Gamma}^{+}_k.
	\end{equation}
\end{definition}
\begin{theorem}\cite[Theorem 1.1]{MR2133413}\label{kconvexalexdandrofftype}
	For $k> n/2$, a $k$-convex function $w \colon B_1 \subset \mathbb{R}^n \to \mathbb{R}$ is punctually second-order differentiable almost everywhere.   
\end{theorem}
Although \eqref{moregeneralL} and \eqref{kconvexequation} are different equations, we observe that when the viscosity solution $w$ to \eqref{moregeneralL} is Lipschitz, \autoref{kconvexalexdandrofftype} can also be applied to $w$.
\begin{theorem}\label{2convexity}
	Let $\Gamma \subset \Gamma_{k}^+$ and $w \in C^{0,1}(B_1)$ be a viscosity solution to \eqref{moregeneralL}.  If either 
	\begin{enumerate}
		\item  $k > n/2$ and $W(x)= \delta_{ij}$ for all $x \in B_1$, or \label{2convexitycon1}
		\item $k=n$,\label{2convexitycon2}
	\end{enumerate}
then $w$ is punctually second-order differentiable almost everywhere in $B_1$. 
\end{theorem}
\begin{proof}
	Pick $C(L,|w|_{0,1;B_1})$ sufficiently large so that
	\begin{equation}\label{ourfunctioniskconvex}
		L(x,p) \leq 2CI \qquad \forall x \in B_1, |p|\leq |w|_{0,1;B_1}.	
	\end{equation}\par 
	Fix arbitrary $x_0 \in B_1$. Pick a paraboloid $\varphi$ such that
	\begin{align}
		\varphi(x_0)&= w(x_0), & \varphi \geq w \text{ in a neighborhood of }x_0.
	\end{align} 
	By \autoref{defnviscositysolution},  
	\begin{equation*}
		\lambda\bigl( W(x_0)\circ (\D^2 \varphi+L(\cdot,\D \varphi) )(x_0) \circ  W(x_0)\bigr) \in \Gamma^{+}_k.
	\end{equation*}
 By \eqref{conecondition} and \eqref{ourfunctioniskconvex},
	\begin{equation}\label{ourfunctioniskconvexyes}
		\lambda\bigl( W(x_0)\circ(\D^2 \varphi +2CI)(x_0) \circ W(x_0) \bigr)  \in \Gamma_k^+.
	\end{equation} 
	In particular, \eqref{ourfunctioniskconvexyes} means 
	\begin{equation}\label{2convexitypart2}
		\lambda\bigl( W(x_0)\circ(\D^2 (\varphi + C |x|^2))(x_0) \circ W(x_0)\bigr) \in\Gamma_k^+.
	\end{equation}\par 
In case when $k=n$. Since $W(x_0) \in \text{Sym}_n^+$, \eqref{2convexitypart2} implies that 
	\begin{equation}\label{2convexitypart2-1}
	\lambda\bigl(\D^2 (\varphi + C |x|^2)(x_0)\bigr) \in\Gamma_n^+.
\end{equation}
Since $x_0$ is arbitrary, \eqref{2convexitypart2-1} implies $w + C |x|^2$ is $n$-convex in the viscosity sense, hence is punctually second-order differentiable almost everywhere by \autoref{kconvexalexdandrofftype}. \par 
In case when $k > n/2$ and $W(x) = \delta_{ij}$ for all $x \in B_1$, since $x_0$ is arbitrary, \eqref{2convexitypart2} already implies that $w + C |x|^2$ is $k$-convex, hence is punctually second-order differentiable almost everywhere by \autoref{kconvexalexdandrofftype}.
\end{proof}
When neither \eqref{2convexitycon1} in \autoref{2convexity} nor \eqref{2convexitycon2} in \autoref{2convexity} is satisfied, as in the case of \autoref{corollary2} and \autoref{corollary3}, we must consider a variant of \autoref{kconvexalexdandrofftype}. \autoref{dualconenonsense2}, \autoref{dualconenonsense}, and \autoref{poincarelikeinequality} are required to prove such a variant. See \autoref{thenewsection} for more. \par 
\begin{definition}\cite{MR2360615}
Let $(\Gamma_k^+)^\ast$ be the dual cone of $\Gamma_k^+$ defined by 
\begin{equation*}
(\Gamma_k^+)^\ast = \{\lambda \in \mathbb{R}^n; \lambda \cdot \mu \geq 0 \; \forall \mu \in \Gamma_k^+\},
\end{equation*}
where 
\begin{equation*}
\lambda \cdot \mu = \sum_i \lambda_i \mu_i  \qquad  \text{ for }  \lambda = (\lambda_1, \cdots, \lambda_n),\quad  \mu= (\mu_1, \cdots, \mu_n).
\end{equation*}\par 
Let $\rho_k$ be the normalized $\sigma_k$ so that for $\lambda \in \Gamma_k^+$,
\begin{equation*}
\rho_k(\lambda) = \biggl(\sigma_k(\lambda) \binom{n}{k}^{-1} \biggr)^{1/k}.
\end{equation*} 
Define $\rho_k^\ast$ on $(\Gamma_k^+)^\ast$ by 
\begin{equation*}
	\rho_k^\ast (\lambda) = \inf \bigl\{  \frac{\lambda \cdot \mu}{n}; \text{ among all } \mu \in \Gamma_k^+ \text{ with }\rho_k(\mu) \geq 1\bigr\}.
\end{equation*}
\end{definition}

\begin{theorem}\cite{MR2133413} \label{dualconenonsense2} We have the following explicit characterization,
	\begin{align*}
		(\Gamma_2^+)^*=\{\lambda \in \overline{\Gamma}_n^+;|\lambda|^2 \leq \frac{1}{n-1}\bigl(\sum_{i=1}^n \lambda_i\bigr)^2\}.
	\end{align*}
	In particular, the following matrices:
	\begin{equation}\label{dualconematricestest}\left\{
		\begin{aligned}
			&\text{the identity matrix } I_n,\\
			&I_n - e_i \otimes e_i, \text{ and }\\
			&I_n + t [e_i \otimes e_j + e_j \otimes e_i] \text{ for } i \neq j \text{ and } 0<t<(n/(2(n-1)))^{1/2},
		\end{aligned}\right.
	\end{equation} all have eigenvalues in $(\Gamma_2^+)^*$. 
\end{theorem}
\autoref{dualconenonsense2} is stated in \cite{MR2133413} without a proof. For the reader's convenience, we have included a proof here.
\begin{proof}
	Let $e=(1,1,\cdots, 1)$.
	Let $\mathcal{C}$ be the cylinder defined by
	\begin{equation*}
		\mathcal{C}= \{x \in \mathbb{R}^n; \text{ distance of }x \text{ to } \{t e; t \in \mathbb{R}\} \text{ equals }1\}.
	\end{equation*}
	Let $\mathcal{S} = \partial \Gamma_1^+ \cap \mathcal{C}$ be a cross section of the cylinder $\mathcal{C}$.  \par 
	Fix $\lambda \in \mathcal{S}$, consider $\lambda + te$ for $t \in \mathbb{R}$. By direct calculation,
	\begin{equation}\label{dualconestructureproperty1}
		\langle \lambda, t e \rangle = \sigma_2(\lambda) + \sigma_2(te) + \sigma_1(\lambda) \sigma_1(te) - \sigma_2(te + \lambda). 
	\end{equation}
	Since $\Gamma_2^+$ lies in $\Gamma_1^+$, we can pick $t>0$, depending on $\lambda$, so that $(t  e + \lambda) \in \partial \Gamma_2^+$. Since $\partial \Gamma_1^+$ is orthogonal to $e$, and $\sigma_2(\mu)=0$ for all $\mu \in \partial\Gamma_2^+$, \eqref{dualconestructureproperty1} reduces to 
	\begin{equation}\label{dualconestructureproperty2}
		0= \sigma_2(\lambda) + t^2 \frac{n(n-1)}{2}.
	\end{equation}
	Note that $\mathcal{S}$ is also the intersection of $\partial \Gamma_1^+$ with $\partial B_1$. Thus $|\lambda|=1$, hence
	\begin{equation}\label{dualconestructureproperty3}
		\sigma_2(\lambda) = \frac{1}{2}\bigl(\sigma_1(\lambda)^2 - |\lambda|^2\bigr) = -\frac{1}{2}.
	\end{equation}
	Using \eqref{dualconestructureproperty3}, \eqref{dualconestructureproperty2} further reduces to 
	\begin{equation}
		t = \sqrt{\frac{1}{n(n-1)}},
	\end{equation} 
	which is independent of $\lambda$. \par 
	We conclude that 
	\begin{equation*}
		\partial\Gamma_2^+ \cap \mathcal{C} = \bigl\{ \lambda + \bigl(n(n-1)\bigr)^{-1/2}  e; \lambda \in \mathcal{S}\bigr\}. 
	\end{equation*}
For two vectors $a$ and $b$, we denote the angle between them by $\angle(a, b)$.	For $ \Gamma_2^+$, we have the following characterization,
	\begin{equation*}
		\begin{split}
		 \Gamma_2^+ &= \bigl\{\lambda \in \mathbb{R}^n; \angle(\lambda, e)<  \frac{\pi}{2}-\arctan( \frac{1}{\sqrt{n  (n-1)}}  |e|)\bigr\}\\
		& = \bigl\{\lambda \in \mathbb{R}^n; \cos \angle(\lambda, e)>  \frac{1}{\sqrt{n}}\bigr\}\\
		&=\bigl\{\lambda \in \mathbb{R}^n; \frac{\lambda \cdot e}{|\lambda| |e| } >  \frac{1}{\sqrt{n}}\bigr\}.
		\end{split}
	\end{equation*}
	Thus $(\Gamma_2^+)^\ast$ has the following characterization,
	\begin{equation*}
		\begin{split}
			&\bigl\{\lambda \in \mathbb{R}^n; \angle(\lambda, \mu) \leq \frac{\pi}{2}, \; \forall \mu \in \Gamma_2^+\bigr\}\\
			=& \bigl\{\lambda \in \mathbb{R}^n; \angle(\lambda, e)\leq  \arctan( \frac{1}{\sqrt{n  (n-1)}} |e|)\bigr\}\\
			=&\bigl\{\lambda \in \mathbb{R}^n; \frac{\lambda \cdot e}{|\lambda|  |e| } \geq \sqrt{1-1/n}\bigr\}\\
			=&\bigl\{\lambda \in \overline{\Gamma}_n^+; \frac{\sigma_1(\lambda)}{|\lambda|  \sqrt{n}} \geq \sqrt{\frac{n-1}{n}}\bigr\}\\
			=& \bigl\{\lambda \in \overline{\Gamma}_n^+; |\lambda|^2 \leq \frac{1}{n-1}\sigma_1(\lambda)^2\bigr\},
		\end{split}
	\end{equation*}
	completing the proof.
\end{proof}
\begin{theorem}\cite[Proposition 2.1]{MR2360615}\label{dualconenonsense}
	Let $A,B \in \text{Sym}_n$ with $	\lambda(A) \in \Gamma_k^+$ and $\lambda(B) \in (\Gamma_k^+)^\ast$, then
	\begin{equation}
		\rho_k(A)  \rho_k^\ast(B) \leq \frac{1}{n}  \textrm{Trace}(AB).
	\end{equation} 
\end{theorem}
For \autoref{dualconenonsense}, also see \cite[Lemma 6.2]{MR0806416} for a proof.
The weighted interior norm and semi-norms on $C^0(\Omega)$ are defined as follows 
\begin{align*}
	|u|_{0;\Omega}^{(\sigma)} &= \sup_{x \in \Omega} d_x^\sigma |u(x)|,\\
	[u]_{0,\alpha;\Omega}^{(\sigma)} &= \sup_{x,y \in \Omega, x \neq y} d_{x,y}^{\sigma+\alpha} \frac{|u(x)-u(y)|}{|x-y|^\alpha},\\
	|u|_{0,\alpha; \Omega}^{(\sigma)} &= [u]_{0,\alpha; \Omega}^{(\sigma)}+|u|_{0;\Omega}^{(\sigma)},
\end{align*}
where $d_x=d(x,\partial \Omega)$ and $d_{x,y}=\min(d_x,d_y)$.
\begin{theorem}\cite[Lemma 2.6]{MR1726702}\label{poincarelikeinequality}
	For any $\epsilon > 0$, for any $u \in C^0(\Omega) \cap L^1(\Omega)$,
	\begin{equation}
		|u|_{0;\Omega}^{(n)} \leq \epsilon^\alpha [u]_{0,\alpha; \Omega}^{(n)} + C\epsilon^{-n} \int_{\Omega} |u|,
	\end{equation}
	for some constant $C$ depending on $n$.
\end{theorem}
\section{$k$-Convex Functions on Riemannian Manifolds}\label{thenewsection}
In this section, we follow the idea in \cite{MR2133413}, which studies $k$-convex functions on Euclidean spaces. By slightly modifying the proof in \cite{MR2133413}, we extend the result to general manifolds assuming some extra conditions. First, we prove a comparison principle. 
\begin{lemma}\label{compatisiontheoryformoregeneralkconvex}
Notations as above. Let $w$ and $v$ be, respectively, a viscosity subsolution and a viscosity supersolution of 
\begin{equation}\left\{\label{generalizedkconvexity}
	\begin{aligned}
		\sigma_k^{1/k}\bigl[\lambda\bigl(W \circ \D^2 u \circ W\bigr)\bigr] &= 0,\\
		\lambda\bigl(W \circ \D^2 u \circ W\bigr) &\in \overline{\Gamma}_k^+ ,
	\end{aligned}\right.
\end{equation}
on open bounded $\Omega \subset \mathbb{R}^n$, with $v \in C^2(\Omega)$. Then $w \leq v$ on $\partial \Omega$ implies that $w \leq v$ in $\Omega$.
\end{lemma}
\begin{proof}
Suppose that $w-v$ has a positive local maximum at $x_0$. Fix $\delta >0$ so that $B_\delta(x_0) \supset \Omega$. For small $\epsilon >0$, $w(x)-v(x)-\epsilon(\delta^2-|x-x_0|^2)$ is $<0$ on $\partial\Omega$, positive at $x_0$, hence has a positive local maximum in $\Omega$. By translating $v+\epsilon(\delta^2-|x-x_0|^2)$ up vertically, we may assume that $v+\epsilon(\delta^2-|x-x_0|^2)$ touches $w$ from above at $x_1 \in \Omega$. By definition of viscosity subsolution, at $x_1$, 
\begin{equation}\label{comparision2}
\lambda\bigl( W \circ \D^2 \bigl(v+\epsilon(\delta^2-|x-x_0|^2)\bigr) \circ W\bigr) \in\overline{\Gamma}_k^+.
\end{equation}
However, \eqref{comparision2} contradicts $v$ being a viscosity supersolution of \eqref{generalizedkconvexity}.
\end{proof}
Next, we show that when $W$ is close enough to $\delta_{ij}$, an estimate similar to \cite[Theorem 2.1]{MR2133413} holds. 
\begin{proposition}\label{generalizedestimate}
Notations as above. Let $w$ be a viscosity subsolution to \eqref{generalizedkconvexity} on open bounded $\Omega \subset \mathbb{R}^n$. For any $\alpha \in \mathbb{R}$ with $2-n/k> \alpha >0$, there exists $\delta$ depending on $\alpha$, $n$, and $k$, such that if $|W-\delta_{ij}|_{0;B_1} \leq \delta$, then for any $\Omega^\prime\Subset \Omega$,
\begin{equation}
	[w]_{0,\alpha;\Omega^\prime}^{(n)} \leq C(n,k,\alpha) \int_{\Omega^\prime} |w|,
\end{equation}
for some constant $C$ depending on $n$, $k$ and $\alpha$.
\end{proposition}
\begin{proof}
Let $2-n/k>\alpha>0$. For $B_R(y) \subset \Omega^\prime$, consider 
\begin{equation*}
	v(x)=\text{osc}_{B_R(y)}( w) \frac{|x-y|^{\alpha}}{R^{\alpha}}.
\end{equation*}
The eigenvalues of the Hessian of $|x-y|^\alpha$ (away from $x=y$) are 
\begin{equation*}
\alpha |x-y|^{\alpha-2}(\alpha-1,1,\cdots, 1).
\end{equation*}
Let 
\begin{equation*}
M_\alpha= \begin{pmatrix}
	\alpha-1 & & &\\
	& 1 & & \\
	& & \ddots & \\
	& & & 1 
\end{pmatrix}.
\end{equation*} 
Since
\begin{equation*}
	\sigma_k\bigl[\lambda\bigl(M_\alpha\bigr)\bigr]= \binom{n-1}{k-1} [n/k-2+\alpha]<0,
\end{equation*} we have $ \lambda(M_\alpha) \in \mathbb{R}^n-\overline{\Gamma}_k^+$. Thus, there exists $\delta$ depending on $\alpha$, $k$, and $n$ such that  
\begin{equation*}
\lambda\bigl(B \circ M_\alpha \circ B \bigr) \in \mathbb{R}^n-\overline{\Gamma}_k^+ \text{ for all }B \in \text{Sym}_n \text{ with }|B-\delta_{ij}| < \delta.
\end{equation*} Therefore, for $|W-\delta_{ij}| < \delta(\alpha,k,n)$,  
\begin{equation*}
\sigma_k\bigl[\lambda \bigl(W \circ \D^2 v\circ W \bigr) \bigr] <0.
\end{equation*}
In other words, $v$ is a viscosity supersolution to \eqref{generalizedkconvexity} in $B_R(y)-\{y\}$. \par 
Since $w(x)-w(y) \leq v(x)$ for $x \in \partial B_R(y)\cup\{y\}$ and $w$ is a viscosity subsolution to \eqref{generalizedkconvexity} by assumption, we have by \autoref{compatisiontheoryformoregeneralkconvex},
\begin{equation}\label{calphaestimaterandomchar}
w(x) - w(y) \leq \text{osc}_{B_R(y)} w  \frac{|x-y|^{\alpha}}{R^{\alpha}},
\end{equation}
for $x \in B_R(y)$. Since $x,y$ are arbitrary, for any $\Omega^\prime \Subset \Omega$, by \eqref{calphaestimaterandomchar},
\begin{equation}
	[w]_{0,\alpha ;\Omega^\prime}^{(n)}\coloneqq \sup_{x,y \in \Omega^\prime} d_{x,y}^{n + \alpha} \frac{|w(x)-w(y)|}{|x-y|^{\alpha}}
	\leq C\sup_{x \in \Omega^\prime} d_x^n |w(x)| = C|w|_{0;\Omega^\prime}^{(n)},
\end{equation}
for $d_x=d(x,\partial \Omega^\prime)$ and $d_{x,y}= \min(d_x,d_y)$.
By \autoref{poincarelikeinequality}, 
\begin{equation}
[w]_{0,\alpha;\Omega^\prime}^{(n)} \leq C(n,k,\alpha)\int_{\Omega^\prime} |w|,
\end{equation}
for some constant $C$ depending on $n$, $k$, and $\alpha$. 
\end{proof}
Since the viscosity subsolution $w$ that we are considering is Lipschitz, if the conditions in \autoref{generalizedestimate} are satisfied and the second-order derivatives of $w$ exist in the distribution sense, then the proof of \cite[Theorem 2.5, Theorem 1.1]{MR2133413} can be applied almost verbatim. Here we include the proof of \autoref{punctuallysecondorderdifferentiablemaing} for completeness. 
\begin{theorem}\cite{MR2133413}\label{punctuallysecondorderdifferentiablemaing}
Notations as above. Let $w$ be a Lipschitz subsolution to \eqref{generalizedkconvexity} and assume there exist signed Borel measures $\mu^{ij}=\mu^{ji}$ such that 
\begin{equation}\label{punctuallyc2conditionsmdistribution}
	\int_{\Omega} w \partial_{ij} \phi(x) \dd{x}= \int_{\Omega} \phi(x) \dd{\mu^{ij}(x)}, \quad \text{ for } i,j = 1,2, \cdots, n,
\end{equation}
for all $\phi \in C_c^\infty (\Omega)$. \par 
If $k > n/2$, then there exists $\delta$ depending on $n$ and $k$, such that $|W-\delta_{ij}|_{0;\Omega} \leq \delta$ implies that $w$ is punctually second-order differentiable almost everywhere in $\Omega$.
\end{theorem}
\begin{proof}
In the following, let $\alpha = (2-n/k)/2>0$ and let $\delta$ be as in \autoref{generalizedestimate} which depends on $\alpha$, $n$, and $k$.
	
Since $w$ is Lipschitz and the Borel measures $\mu^{ij}$ satisfy \eqref{punctuallyc2conditionsmdistribution}, it follows that $w$ is differentiable almost everywhere and $\partial_i w$ is of local bounded variation. \par 
By the Lebesgue decomposition theorem, write 
\begin{equation*}
\dd{\mu^{ij}}= \D^2_{ij} w \dd{x} + [\dd{\mu^{ij}}]_{\text{s}},
\end{equation*}
where $\D^2_{ij}w$ and $[\dd{\mu^{ij}}]_{\text{s}}$ represent the absolutely continuous and singular parts of $\mu^{ij}$ with respect to the Lebesgue measure, respectively. 
Then for almost every $x \in \Omega$, 
\begin{equation}\label{punctuallysecondorderdifferentiablecondition}\left\{
\begin{aligned}
\lim_{r \to 0} \fint_{B_r(x)} |\D w(x)-\D w(y)| \dd{y}&=0,\\
\lim_{r \to 0} \fint_{B_r(x)} |\D^2 w(x)-\D^2 w(y)| \dd{y}&=0,\\
\lim_{r \to 0} \frac{|[\dd{\mu^{ij}}]_\text{s}|(B_r(x))}{r^n}&=0.\\
\end{aligned}\right.
\end{equation}
By the proof of \cite[Theorem 6.9]{MR3409135}, at points $x$ where \eqref{punctuallysecondorderdifferentiablecondition} is satisfied, 
\begin{equation}\label{integralsecondordertaylorexpansion}
\fint_{B_r(x)}|w(y)-w(x)-\D w(x) \cdot (y - x)-\frac{1}{2} (y - x)^\intercal  (\D^2 w(x))  (y - x)| \dd{y}= o(r^2).
\end{equation}\par 
For simplicity, assume that $x=0$ and take 
\begin{equation*}
	h(y)=w(y)-w(0)-\D w(0) \cdot y-\frac{1}{2} y^\intercal ( \D^2 w(0)) y.
\end{equation*}
Let $\Lambda = |\D^2 w(0)|$. Then 
\begin{equation}
h(y)+ \Lambda |y|^2= \bigl(w(y)-w(0)-\D w(0) \cdot y\bigr) + \bigl( \Lambda |y|^2 -\frac{1}{2} y^\intercal (\D^2 w(0))  y\bigr),
\end{equation}
is the sum of a viscosity subsolution to \eqref{generalizedkconvexity} and a convex function, hence is a viscosity subsolution to \eqref{generalizedkconvexity}. Therefore 
\begin{equation}\label{generalizedinequality}
\begin{split}
\sup_{y,z\in B_r(0)} \frac{|h(y)-h(z)|}{|y-z|^\alpha} &\leq \sup_{y,z\in B_r(0)} \frac{|(h(y)+ \Lambda |y|^2)-(h(z)+ \Lambda |z|^2)|}{|y-z|^\alpha}\\
& \quad  + 4 \Lambda r^{2 -\alpha}\\
&\leq C(n,k)  \frac{1}{r^\alpha}\fint_{B_{2r}} |h(y_0) + \Lambda |y_0|^2|\dd{y_0} + 4 \Lambda r^{2-\alpha}\\
& \leq C(n,k,\Lambda)  \biggl(\frac{1}{r^\alpha}\fint_{B_{2r}} |h(y_0)|\dd{y_0} + r^{2-\alpha}\biggr),
\end{split}
\end{equation}
where the second inequality follows by applying \autoref{generalizedestimate} to $h(y)+\Lambda |y|^2$ on $B_{2r}(0)$.\par 
From \eqref{integralsecondordertaylorexpansion}, for any $\epsilon, \delta$ with $0 <\epsilon <1$ and $0< \delta^{1/n} <1/2$, there is $r_0(\epsilon,\delta)$ such that for any $r < r_0$ and $y \in B_{r/2}(0)$,
\begin{equation}\label{integralsecondordertaylorexpansionestimate}
|z \in B_r(y); |h(z)|\geq \epsilon r^2| \leq \frac{1}{\epsilon r^2}\int_{B_{2r}(0)}|h(z)|\dd{z} =\frac{1}{\epsilon} o(r^n) < \delta |B_r(y)|.
\end{equation}
From \eqref{integralsecondordertaylorexpansionestimate}, for each $y \in B_{r/2}(0)$, there exists $z \in B_r(y)$ such that 
\begin{align*}
|h(z)| &< \epsilon r^2,  & |y-z|&\leq \delta^{1/n} r.
\end{align*}
Then by \eqref{generalizedinequality},
\begin{equation*}
\begin{split}
	|h(y)| &\leq |h(z)| + |h(y)- h(z)|\\
	& \leq \epsilon r^2+ C |y-z|^\alpha \biggl(\frac{1}{r^\alpha}\fint_{B_{2r}} |h(y_0)|\dd{y_0} + r^{2-\alpha}\biggr)\\
	& \leq \epsilon r^2 + C\delta^{\alpha/n} \biggl(\fint_{B_{2r}} |h(y_0)|\dd{y_0} + r^{2}\biggr)\\
	&= r^2(\epsilon + C \delta^{\alpha/n}) + o(r^2),
\end{split}
\end{equation*}
where the $o(r^2)$ term in the last expression is independent of $\epsilon$ and $\delta$. Since $\epsilon$ and $\delta$ can be taken arbitrarily small, $\sup_{B_r(0)}|h(y)|=o(r^2)$, completing the proof. 
\end{proof}

In the rest of this section, we demonstrate that \eqref{punctuallyc2conditionsmdistribution} is a natural condition to work with. In fact, the viscosity solutions in \autoref{corollary2} and \autoref{corollary3} satisfy \eqref{punctuallyc2conditionsmdistribution}. First, we need the following technical lemma. For $0 \leq \tau <1$, any smooth solution $w_\tau$ to \eqref{modifiedschouten} satisfies \eqref{notinterestingtechnicallemmapt1}. 
\begin{lemma}\label{notinterestingtechnicallemma}
Notations as above.
Let $w_\tau \in C^\infty(\Omega)$ be a family of functions for $\tau \in [0,1)$, such that 
\begin{equation} \label{notinterestingtechnicallemmapt1} 
	\begin{aligned}
	\lambda\big(\sqrt{g^{-1}} \circ (\D_{ij}^2 w_\tau&+\frac{1-\tau}{n-2}(\Delta w_\tau)g_{ij} +\frac{2-\tau}{2} g^{sl}(w_\tau)_s(w_\tau)_l  g_{ij}\\
	&-\Gamma_{ij}^l (w_\tau)_l- (w_\tau)_i \otimes (w_\tau)_j-(H_\tau)_{ij})  \circ \sqrt{g^{-1}} \bigr) \in \Gamma_k^+,
	\end{aligned}
\end{equation}
where the metric $g$ is understood to be a map from $\overline{\Omega}$ to $\text{Sym}_n^+$, $\Gamma_{ij}^l$ are the Christoffel symbols associated to $g$, and $H_\tau \colon \overline{\Omega} \to \text{Sym}_n$ with $\sup_{\tau\in [0,1)} \Vert H_\tau \Vert_{C^0(\overline{\Omega})} < \infty$. \par 
If $|w_\tau|_{0,1;\Omega} < K$ for all $\tau \in [0,1)$, then there exists $C$ independent of $\tau$ such that 
\begin{equation*}
\begin{aligned}
\lambda\bigl(\sqrt{g^{-1}} \circ \D^2 (w_\tau + C |x|^2) \circ \sqrt{g^{-1}}+ \frac{1-\tau}{n-2}\Delta(w_\tau+C|x|^2)  I \bigr) &\in \Gamma_k^+.
\end{aligned}
\end{equation*}
\end{lemma}
\begin{proof}
Since the family $w_\tau$ is uniformly Lipschitz, there exists $C$ independent of $\tau$ such that 
\begin{equation}\label{Lipschitzcovertheorem}
	\frac{2-\tau}{2} g^{sl}(w_\tau)_s(w_\tau)_l  g_{ij} -\Gamma_{ij}^l (w_\tau)_l- (w_\tau)_i \otimes (w_\tau)_j-H_\tau \leq C  \delta_{ij},
\end{equation}
where \eqref{Lipschitzcovertheorem} holds in the matrix sense. Hence the claim follows.
\end{proof}
\autoref{checkconditionnaturalgood} is similar to \cite[Theorem 2.4]{MR2133413}, which studies the case when $W(x)=\delta_{ij}$ for all $x \in \Omega$. We combat the disturbance due to $W$ using the Lipschitz regularity of $w$. 
\begin{theorem}\label{checkconditionnaturalgood}
Notations as above. Let $w_\tau \in C^\infty (\Omega)$ be a family of functions for $\tau \in [0,1)$ such that  
\begin{equation*}
		\lambda\bigl(W \circ \D^2 w_\tau \circ W + \frac{1-\tau}{n-2}\Delta(w_\tau) I\bigr) \in \Gamma_k^+.
\end{equation*}
If $w_\tau$ converges uniformly to $w$ on $\Omega$ and $|w_\tau|_{0,1;\Omega} \leq C < \infty$ for all $\tau \in [0,1)$, then there exist signed measures $\mu^{ij}$ such that 
\begin{equation*}
\int_{\Omega} w(x)  \partial_{ij} \phi \dd{x}=\int_{\Omega} \phi \dd{\mu^{ij}}
\end{equation*}
for all $\phi \in C_c^\infty(\Omega)$.
\end{theorem}
\begin{proof}
We follow the idea of \cite[Theorem 2.4]{MR2133413}.\par 
\noindent \textbf{Step 1.}	For $B= (B^{ij})\in \text{Sym}_n$, define the distribution $T_B \colon C_c^\infty(\Omega) \to \mathbb{R}$ by
	\begin{equation*}
		T_B(\phi) \coloneqq \int_{\Omega} w(x)  \sum_{ijsl}B^{ij} W_{is} \partial_{sl}\phi  W_{lj}\dd{x}.
	\end{equation*}
	Then 
		\begin{equation}\label{2nddifferentiableindistributionsensept1}
		T_B(\phi) = \lim_{\tau \to 1}\int_{\Omega} w_\tau(x)  \sum_{ij}B^{ij} \bigl(\sum_{sl} W_{is} \partial_{sl}\phi  W_{lj} + \delta_{ij}  \frac{1-\tau}{n-2} \Delta \phi\bigr)\dd{x}.
	\end{equation}
Using integration by parts, \eqref{2nddifferentiableindistributionsensept1} becomes 
\begin{equation*}
\begin{aligned}
	T_B(\phi) &= \lim_{\tau \to 1}\int_{\Omega} \sum_{ij}B^{ij}  \bigl[\bigl( \sum_{sl}W_{is} \partial_{sl}w_\tau W_{lj}+ \delta_{ij} \frac{1-\tau}{n-2}\Delta w_\tau\bigr) \bigr]  \phi\dd{x} \\
	&\quad + \lim_{\tau \to 1} \int_{\Omega} \sum_{ijsl} B^{ij}  \partial_s(W_{is}  W_{lj})  \partial_l w_\tau \phi \dd{x} \\
	&\quad + \lim_{\tau \to 1} \int_{\Omega} \sum_{ijsl} B^{ij}  \partial_l(W_{is}  W_{lj})  \partial_s w_\tau  \phi \dd{x}\\
	&\quad + \lim_{\tau \to 1} \int_{\Omega} \sum_{ijsl} B^{ij}  \partial_{sl}(W_{is}  W_{lj})  w_\tau  \phi \dd{x}.
\end{aligned}	
\end{equation*}\par 
 Fix $\lambda(B) \in (\Gamma_2^+)^*$. Since $|w_\tau|_{0,1;\Omega}$ is bounded independent of $\tau$, we can choose $C(|w|_{0,1;\Omega}, W,B)$ large enough so that for $\phi \geq 0$,  
\begin{equation}\label{2nddifferentiableindistributionsensept3}
	\begin{multlined}
T_B(\phi) +C(|w|_{0,1;\Omega}, W,B)    \int_{\Omega} \phi \dd{x}\\
\geq \lim_{\tau\to 1} \int_{\Omega} \sum_{ij} B^{ij} \bigl( \sum_{sl}W_{is} \partial_{sl}w_\tau W_{lj}+ \delta_{ij} \frac{1-\tau}{n-2}\Delta w_\tau\bigr) \phi \geq 0,
	\end{multlined}
\end{equation}
where the last inequality follows from \autoref{dualconenonsense} and the assumption that 
\begin{equation*}
	\lambda \bigl( \sum_{sl}W_{is} \partial_{sl}w_\tau W_{lj}+ \delta_{ij} \frac{1-\tau}{n-2}\Delta w_\tau\bigr) \in \Gamma_k^+.
\end{equation*}\par 
In particular, the left-hand side of \eqref{2nddifferentiableindistributionsensept3} is a positive linear functional on $C_c^\infty (\Omega)$, and the Riesz Representation Theorem \cite{MR0924157} applies. We conclude that there is a signed Borel measure $\mu^B$ such that 
\begin{equation}\label{2nddifferentiableindistributionsensept4}
T_B(\phi) = \int_{\Omega} \phi \dd{\mu^B}. 
\end{equation}
\noindent \textbf{Step 2.} Plugging the matrices in \eqref{dualconematricestest} into \eqref{2nddifferentiableindistributionsensept4}, we obtain some signed Borel measures. By \autoref{dualconenonsense2}, after summing and taking the difference between these measures, we can get a collection of signed Borel measures $\mu^{ij}$ such that
\begin{equation*}
\sum_{sl} \int_{\Omega} w(x)  W_{is} \partial_{sl} \phi(x)  W_{lj} \dd{x}= \int_{\Omega}\phi \dd{\mu^{ij}}. 
\end{equation*} 
Define a new signed Borel measure $\widetilde{\mu}^{ij}$ by 
\begin{equation*}
\widetilde{\mu}^{ij}=\sum_{nm} W^{ni} W^{mj}  \mu^{nm},
\end{equation*}
where $W^{ij}$ is the inverse of $W_{ij}$. Then for $\phi \in C_c^\infty(\Omega)$, since $W$ is smooth on $\Omega$ and $w$ is Lipschitz, using integration by parts,
\begin{align*}
\int_{\Omega} \phi \dd{\widetilde{\mu}^{ij}}&= \sum_{nm} \int_{\Omega}  W^{ni}  \phi  W^{mj}\dd{\mu^{nm}}\\
&=\sum_{slnm}\int_{\Omega} w  W_{ns} \partial_{sl} ( W^{ni}  \phi  W^{mj}) W_{lm} \dd{x}\\
&= \int_{\Omega} w  \partial_{ij} \phi \dd{x}\\
& \quad + \sum_{slnm}\int_{\Omega} w(x)  W_{ns} W_{lm} \partial_{sl} ( W^{ni}  W^{mj})  \phi(x) \dd{x}\\
& \quad - \sum_{slnm}\int_{\Omega} \partial_l\bigl(w(x)  W_{ns} W_{lm} \partial_{s} ( W^{ni}  W^{mj})\bigr)   \phi(x) \dd{x}\\
& \quad - \sum_{slnm}\int_{\Omega} \partial_s \bigl(w(x) W_{ns} \partial_{l} ( W^{ni}  W^{mj}) W_{lm}\bigr)  \phi(x) \dd{x}\\
&\eqqcolon \int_{\Omega} w(x)  \partial_{ij} \phi \dd{x} + \int_{\Omega} \phi \dd{\overline{\mu}^{ij}},
\end{align*}
where $\overline{\mu}^{ij}$ is some signed Borel measure. Hence we have
\begin{equation*}
\int_{\Omega} \phi \dd{(\widetilde{\mu}^{ij} - \overline{\mu}^{ij})}= \int_{\Omega} w(x)  \partial_{ij} \phi \dd{x},
\end{equation*}
completing the proof.
\end{proof}
We conclude this section with a combination of \autoref{punctuallysecondorderdifferentiablemaing}, \autoref{notinterestingtechnicallemma}, and \autoref{checkconditionnaturalgood}.
\begin{theorem}\label{punctuallydifferentiablealmosteverysmoothapproximation}
Let $(M,g)$ be a Riemannian manifold. Let $w_\tau \in C^\infty (M)$ be a family of functions for $0 \leq \tau < 1$, such that
\begin{equation*}\left\{
\begin{aligned}
	|w_\tau|_{0,1;M} \leq C< \infty &\text{ for all }\tau \in [0,1),\\
	\lambda(-g_{w_{\tau}}^{-1}A_{g_{w_{\tau}}}^{\tau})\in\Gamma_{k}^{+}&,
\end{aligned}\right.
\end{equation*}
and $w_\tau$ converges uniformly to $w \in C(M)$ on compact subsets of $M$.\par 
If $k>n/2$, then $w$ is punctually second-order differentiable almost everywhere on $M$.
\end{theorem}
\begin{proof}
Let $\delta(n,k)$ be as in \autoref{punctuallysecondorderdifferentiablemaing}. Fix any $p \in M$. Take a normal coordinate system $(x^1,\cdots, x^n)$ centered at $p$, then $g_{ij}(0)= g_p(\partial_{x^i}, \partial_{x^j})=\delta_{ij}$. In a neighborhood $U$ of $0$, $|\sqrt{g^{-1}}-\delta_{ij}|_{0;U} \leq \delta(n,k)$, where $g^{-1}(x)$ denotes the inverse of the matrix $g_{ij}(x)$.
	
By \autoref{notinterestingtechnicallemma}, there is a constant $C > 0$ independent of $\tau$ such that
\begin{equation}
	\lambda\bigl(\sqrt{g^{-1}} \circ \D^2 (w_\tau + C |x|^2) \circ \sqrt{g^{-1}}+ \frac{1-\tau}{n-2}\Delta(w_\tau+C|x|^2)  I \bigr) \in \Gamma_k^+,
\end{equation}
for all $\tau \in [0,1)$.
Since $w_\tau + C |x|^2$ converges uniformly to $w + C |x|^2$ on $U$ with \begin{equation}
	\sup_{\tau \in [0,1)}|w_\tau + C |x|^2|_{0,1;U}<\infty,
\end{equation} \autoref{checkconditionnaturalgood} implies the existence of signed measures $\mu^{ij}$ such that 
\begin{equation}
\int_{U} (w+C|x|^2) \partial_{ij} \phi \dd{x}=\int_{U} \phi \dd{\mu^{ij}}
\end{equation}
for all $\phi \in C_c^\infty(U)$. By \autoref{punctuallysecondorderdifferentiablemaing}, $w + C|x|^2$, and hence $w$, is punctually second-order differentiable almost everywhere on $U$. Since $p$ is arbitrary, $w$ is punctually second-order differentiable almost everywhere on $M$, completing the proof.
\end{proof}
\section{Local Regularity}\label{localregularitysection}
\autoref{C2alpha} is our main theorem on the regularity of viscosity solutions to \eqref{moregeneralL}.
\begin{theorem}\label{C2alpha}
	Let $w$ be a viscosity solution to \eqref{moregeneralL} in $B_1$. If $w$ is punctually second-order differentiable at $0$,
	then $w$ is $C^{2,\alpha}$ in a neighborhood of $0$.
\end{theorem}

The proof strategy of \autoref{C2alpha} is the following.
We construct $C^{2,\alpha}$ local solutions to \eqref{moregeneralL} that approximate the viscosity solution $w$ near points where $w$ is punctually second-order differentiable. These $C^{2,\alpha}$ local solutions are required to construct operators so that \eqref{SavinP0}, \eqref{SavinP1}, \eqref{SavinP2.5}, and \eqref{SavinP2} in \autoref{Savinisgod} apply. We show that for some of these operators, \autoref{Savinisgod} applies. See \autocite[Theorem 1.3]{MR4611402} for a similar argument.
\begin{theorem}\label{inversemappingsigma2}
Notations as above. Let $P$ be a paraboloid such that 
	\begin{equation}\label{inversemappingstartingassumption}\left\{
	\begin{aligned}
	f\bigl[\lambda\bigl( W \circ (\D^2 P + L(\cdot ,\D P) ) \circ W \bigr)\bigr](0)=& R(0) e^{2P(0)}, \\
	\lambda\bigl( W \circ (\D^2 P + L(\cdot ,\D P))  \circ W \bigr)(0) \in& \Gamma.
\end{aligned}\right.
	\end{equation}
Then there are constants $C_1,C_2$ depending on $f,P,\alpha,R,L$, and $W$ such that for any $r^{1-\alpha} \leq C_1$, there is $w^{(r)} \in C^{2,\alpha}(\overline{B}_1) \cap C^{4, \alpha}(B_{1/2})$ satisfying
\begin{equation}\label{propertitysofWupr}\left\{
	\begin{alignedat}{2}
	f\bigl[\lambda\bigl( W_r \circ (\D_y^2 w^{(r)} + L^r(\cdot  ,\D_y w^{(r)}))  \circ W_r \bigr)\bigr](y) &=  R_r e^{2w^{(r)}}(y) \quad  &&\text{ in } B_1,\\
	\lambda\bigl( W_r \circ (\D^2_y w^{(r)} + L^r(\cdot  ,\D_y w^{(r)}) )\circ  W_r \bigr)(y) &\in \Gamma,\\
|w^{(r)} - P_r|_{2,\alpha;B_1},  |\D_y^3 (w^{(r)} - P_r)|_{0,\alpha;B_{1/2}}, |\D_y^4 (w^{(r)} - P_r)|_{0,\alpha;B_{1/2}}   &\leq  r^{2+\alpha}  C_2,\\
	w^{(r)}&= P_r \quad && \text{ on } \partial B_1,
\end{alignedat}\right.
\end{equation}
where
\begin{equation}\label{blowup}\left\{
	\begin{aligned}
		L^r(y,p) &= r^2L_0(ry,p)+ rL_1(ry,p)+ L_2(ry,p),\\
		P_r(y)&=P(ry)+\ln r, \\
		R_r(y) &= R(ry),\\
	W_{r}(y) & = W(ry),
	\end{aligned}\right.
\end{equation}
 for $y \in \overline{B}_1$ and $0 < r \leq 1$.
\end{theorem}
\begin{proof}
Let $P$ be the paraboloid in the theorem. Define $L^r$, $P_r$, $R_r$, and $W_r$ by \eqref{blowup}. We also adopt the coordinate transformation 
\begin{align}\label{coordinatetransformation}
	y&=x/r, & y &\in B_1, & x \in B_r.
\end{align}\par 
 Define $\mathcal{F}^r \colon  C^{2,\alpha}(\overline{B}_1) \cap C_0 (\overline{B}_1) \to C^\alpha (\overline{B}_1)$ by
\begin{equation*}
	\mathcal{F}^r[u] =
	f\bigl[\lambda\bigl( W_{r}(y) \circ (\D_y^2 (u+P_r) + L^r(\cdot  ,\D_y (u+P_r)) ) \circ W_{r}(y)\bigr)\bigr] - R_re^{2(u+P_r)},
\end{equation*}
where $\D_y$ are derivatives taken with respect to the $y$-coordinates.\par
We use the inverse mapping theorem to construct a solution. \par 
\noindent \textbf{Step 1.} We measure the $C^\alpha$ norm of $\mathcal{F}^r[0]$, where 
\begin{equation}\label{F0expression}
\mathcal{F}^r[0]= f\bigl[\lambda\bigl( W_{r} \circ (\D_y^2 P_r+ L^r(\cdot, \D_y P_r))\circ  W_{r}\bigr)(y)\bigr]- R_r e^{2P_r}.
\end{equation}
 For any function $u \in  C^{2,\alpha}(\overline{B}_1)$, define $u_r \in  C^{2,\alpha}(\overline{B}_1)$ by 
\begin{equation*}
u_r(y)=u(x).
\end{equation*} 
	By direct calculation, we have 
	\begin{equation}\label{scalingforFr}
		\mathcal{F}^r[u_r](y) = r^{2} \mathcal{F}^1[u](x),
	\end{equation}
	and
	\begin{equation*}
		\sup_{y_1, y_2} \frac{|\mathcal{F}^r[u_r](y_1) - \mathcal{F}^r[u_r](y_2)|}{|y_1 - y_2|^{\alpha}}\leq r^{2+\alpha}\sup_{x_1, x_2} \frac{|\mathcal{F}^1[u](x_1) - \mathcal{F}^1[u](x_2)|}{|x_1 - x_2|^{\alpha}}.
	\end{equation*}\par
	By \eqref{inversemappingstartingassumption} and \eqref{F0expression}, $\mathcal{F}^1[0](0)=0$ and $\mathcal{F}^1[0]$ is $C^1$, we have
	\begin{equation*}
		|\mathcal{F}^{r}[0]|_{0;B_1} \leq r^{2} |\mathcal{F}^1[0]|_{0;B_r} \leq r^{3} C_0,
	\end{equation*}
	and 
	\begin{equation*}
		[\mathcal{F}^{r}[0]]_{0,\alpha;B_1} \leq r^{2+\alpha} [\mathcal{F}^1[0]]_{0,\alpha;B_{r}} \leq r^{3}C_0,
	\end{equation*}
	hence
	\begin{equation}\label{F_r[P_r]decreaserate}
		|\mathcal{F}^{r}[0]|_{0,\alpha;B_1} \leq r^{3} C_0,
	\end{equation}
	for some constant $C_0$ depending on $P$.\par  
\noindent \textbf{Step 2.} We study the Fr\'{e}chet derivative of $\mathcal{F}^{r}$ at $0$. Let $a^{ij}_r,b^i_r,c_r$ be the coefficients of $D \mathcal{F}^r[0]$, so that
	\begin{equation*}
	\D \mathcal{F}^r[0](h) = a^{ij}_r(y) h_{ij}+ b_r^i h_i+ c_r h,
	\end{equation*} where
\begin{equation}\label{abcrdef}\left\{
	\begin{aligned}
	a^{ij}_r(y) &= \sum_{kl}\biggl(\pdv{}{M_{kl}} (f\circ \lambda)\biggr) (W_{r})_{li}  (W_{r})_{jk}, \\
	b_r^s(y)&=\sum_{ij}a^{ij}_r(y)  (\partial_{p^s}L_{ij}^{r}) (y, \D_y P_r),  \\
	c_r(y)&=- 2R_r e^{2P_r}.
\end{aligned}\right.
\end{equation}The subscript $r$ in \eqref{abcrdef} does not refer to a rescaling in $r$.\par 
By \eqref{conecondition}, $\pdvs{f}{\lambda_i}$ is homogeneous of degree zero, and due to the conditions \eqref{structurealconditionsoinlnL}, \eqref{blowup} on $L$ and $L^r$ respectively, we have
	\begin{align}\label{abcblowup}
		a^{ij}_r(y)&=  a^{ij}(x), &
		b^{i}_r(y)&= r b^{i}(x), &
		c_r(y)&= r^{2} c(x),
	\end{align} 
	hence,
	\begin{equation}\label{F_rcoefficientblowup} \left\{
	\begin{aligned}
	|a^{ij}_r|_{0,\alpha;B_1} &\leq |a^{ij}|_{0,\alpha;B_1} , \\
	|b^{i}_r|_{0,\alpha;B_1} &\leq r |b^{i}|_{0,\alpha;B_1} ,\\
	|c_r|_{0,\alpha;B_1} &\leq  r^{2}|c|_{0,\alpha;B_1}. 
\end{aligned}\right.
	\end{equation}
	Define
	\begin{align*}
		\Lambda_r &= |a^{ij}_r|_{0,\alpha;B_1}+ |b^{i}_r|_{0,\alpha;B_1}+|c_r|_{0,\alpha;B_1}, & a^{ij}_r \xi_i \xi_j \geq \lambda_r |\xi|^2.
	\end{align*}
	By \eqref{abcblowup}, \eqref{F_rcoefficientblowup}
	\begin{align*}
		\Lambda_r &\leq \Lambda, & \lambda_r&= \lambda.
	\end{align*}
	
 For any $h \in C^{2,\alpha}(\overline{B}_1) \cap C_0(\overline{B}_1)$, let $w = D\mathcal{F}^r[0](h) \in C^\alpha(\overline{B}_1)$. Since $h =0$ on the boundary $\partial B_1$, and $c_r$ converges to $0$ as $r \to 0$, following the existence and uniqueness theory of elliptic PDE with Dirichlet boundary condition \cite[Corollary 3.8, Theorem 6.8]{MR0473443}, $D\mathcal{F}^r[0]$ is invertible for $r$ sufficiently small. Note that in \autoref{inversemappingsigma2}, $c_r$ is already negative by assumption since $R$ is positive; hence $D\mathcal{F}^r[0]$ is invertible for  $r$ sufficiently small. The argument above is used for the proof of \autoref{variantinversemappingsigma2}, which is a variant of \autoref{inversemappingsigma2} where the corresponding $c_r$ is no longer negative.\par  In particular, the boundary value problem,
 \begin{equation}\label{frechletderivativeFr}\left\{
	\begin{alignedat}{2}
	a_r^{ij}  h_{ij} +  b_r^i  h_i +  c_r h &=  w \quad &&  \text{ in } B_1,\\
	h&=0 \quad &&  \text{ on } \partial B_1,
\end{alignedat}\right.
 \end{equation}
	has a solution in $C^{2,\alpha}(\overline{B}_1)$.  \eqref{frechletderivativeFr} is uniformly elliptic independent of $r$. Thus by classical Schauder estimate \cite[Theorem 6.6]{MR0473443}, 
	\begin{equation*}
		\begin{split}
			|h|_{2,\alpha;B_1} &\leq C_3(n,\alpha, \lambda,\Lambda) \bigl(|h|_{0;B_1}+|w|_{0,\alpha;B_1} \bigr)\\
			& \leq C_4(n,\alpha, \lambda, \Lambda)   |w|_{0,\alpha;B_1},
		\end{split}
	\end{equation*}
where the second inequality comes from \cite[Corollary 3.8]{MR0473443}.\par 
	It follows that
	\begin{equation}\label{DF_rdecreaserate}
		\Vert\D\mathcal{F}^r[0]^{-1}\Vert  \leq  C_4.
	\end{equation}
	By \eqref{F_rcoefficientblowup}, we also have 
	\begin{equation}\label{DF_rdecreaserate1}
		\Vert \D\mathcal{F}^r[0]\Vert  \leq C_5 \Lambda,
	\end{equation} for some constant $C_5$ independent of $r$,
	hence 
	\begin{equation}\label{DF_rdecreaseratelb}
		\Vert \D\mathcal{F}^r[0]^{-1}\Vert  \geq \frac{1}{\Vert \D\mathcal{F}^r[0]\Vert} \geq C_5^{-1}\Lambda^{-1}.
	\end{equation}
	Therefore by \eqref{DF_rdecreaserate} and \eqref{DF_rdecreaseratelb}, $\Vert \D \mathcal{F}^r[0]^{-1}\Vert $ is bounded from above and below independent of $r$.
	\par 
\noindent \textbf{Step 3.} We want to study the Fr\'{e}chet derivative of $\mathcal{F}^{r}$ at functions near $0$.
Let $a^{ij}_u, b^i_u, c_u$ denote the coefficients of $D\mathcal{F}^1[u]$ for $u \in C^{2,\alpha}(\overline{B}_1)\cap C_0(\overline{B}_1)$, such that 
\begin{equation*}
	\D \mathcal{F}^1[u](h)= a_u^{ij}h_{ij}+b^i_u h_i+ c_u h.
\end{equation*}
	By continuity of $L$, there is $\delta > 0$ such that for any $u_1, u_2 $ with $|u_i|_{2,\alpha; B_1}< \delta$, $i=1,2$, there is
	\begin{equation*}
		|a^{ij}_{u_1} - a^{ij}_{u_2}|_{0,\alpha;B_1} + |b^{i}_{u_1} - b^{i}_{u_2}|_{0,\alpha;B_1}+ |c_{u_1} - c_{u_2}|_{0,\alpha;B_1} \leq   \frac{1}{2}C_4^{-1},
	\end{equation*}
	and hence,
	\begin{equation} \label{Dfbeforeblowup}
		\Vert \D\mathcal{F}^1[u_1]-D\mathcal{F}^1[u_2]\Vert \leq \frac{1}{2}C_4^{-1}.
	\end{equation}\par 
	We claim that there exists some universal $C_6 > 0$, such that for any $v_1,v_2$ with 
	\begin{equation}\label{v1v2condition}
	|v_i|_{2,\alpha;B_1} < r^{2+\alpha} \delta/C_6,	\end{equation}
 and $r < 1/3$, then 
	\begin{equation} \label{inversemappingneighborhood}
		\Vert \D\mathcal{F}^r[v_1]-D\mathcal{F}^r[v_2]\Vert  \leq  \frac{1}{2} C_4^{-1}.
	\end{equation}\par 
	To see this, by \cite[Lemma 6.37]{MR0473443}, we can first extend $v_i$ to a function $\widetilde{v}_i$ in $C_0^{2,\alpha}(B_2)$ such that 
	\begin{equation}\label{GTextensionlemma}
		|\widetilde{v}_i|_{2,\alpha;B_2} \leq C_7 |v_i|_{2,\alpha;B_{1}},
	\end{equation}
	for some universal constant $C_7$, and take 
	\begin{equation*}
		u_i(x)=
		\begin{cases}
			\widetilde{v}_i(r^{-1}x) & \text{ if } |x|< 2r\\
			0  & \text{ if }2r \leq |x|\leq 1.	
		\end{cases}
	\end{equation*}
Then $u_i \in C^{2,\alpha}(\overline{B}_1) \cap C_0(\overline{B}_1)$, and  
\begin{align} \label{uvblowuprelation}
v_i(y)&=u_i(x), & x &=ry.
\end{align}
	Take $C_6 =C_7$, by \eqref{GTextensionlemma} and \eqref{v1v2condition}
	\begin{equation}\label{u_idistance}
		|u_i|_{2,\alpha;B_1} \leq r^{-2 -\alpha} |\widetilde{v}_i|_{2,\alpha;B_2} \leq C_7 r^{-2 - \alpha} |v_i|_{2,\alpha;B_1} \leq \delta.
	\end{equation}
 By \eqref{u_idistance}, \eqref{uvblowuprelation}, our choice of $\delta$ and \eqref{Dfbeforeblowup},
	\begin{align*}
\Vert D \mathcal{F}^1[u_1]- D \mathcal{F}^1[u_2]\Vert  &\leq \frac{1}{2} C_4^{-1}  & &\Longrightarrow & \Vert D \mathcal{F}^r[v_1]- D \mathcal{F}^r[v_2]\Vert &\leq \frac{1}{2} C_4^{-1} ,
\end{align*}
where the implication follows from a calculation similar to \eqref{F_rcoefficientblowup} and \eqref{DF_rdecreaserate1}, with $\D \mathcal{F}^1[u_i]$ in place of $\D \mathcal{F}^1[0]$ and $\D \mathcal{F}^r[v_i]$ in place of $\D \mathcal{F}^r[0]$ respectively.\par 
\noindent \textbf{Step 4.}	We check that the inverse mapping theorem applies for all $r$ sufficiently small. The contraction map for our target $0$ function is given by 
	\begin{equation*}
		T_r[v]= \bigl(D\mathcal{F}^r[0]\bigr)^{-1}\bigl[ -\mathcal{F}^r[v]+D\mathcal{F}^r[0](v)\bigr].
	\end{equation*}
	For any $v_1, v_2$ with $|v_i|_{2,\alpha;B_1} \leq r^{2+\alpha} \delta/C_6$, we have by \eqref{DF_rdecreaserate} and \eqref{inversemappingneighborhood},
	\begin{equation}\label{contractivemapping}
	\begin{split}
	&\frac{|T_r[v_1]-T_r[v_2]|_{2,\alpha;B_1}}{|v_1-v_2|_{2,\alpha;B_1}}\\
	&\quad \leq  \frac{\lVert \D\mathcal{F}^r[0]^{-1}\rVert }{|v_1-v_2|_{2,\alpha;B_1}}\bigl\vert \mathcal{F}^r[v_2]-\mathcal{F}^r[v_1]-\D\mathcal{F}^r[0](v_2-v_1)\bigr\vert_{0,\alpha;B_1}\\
	&\quad \leq  \frac{\Vert \D\mathcal{F}^r[0]^{-1}\Vert  }{|v_1-v_2|_{2,\alpha;B_1}} \biggl\vert\int_0^1\bigl( D\mathcal{F}^r[t(v_2)+(1-t)v_1]-D\mathcal{F}^r[0]\bigr)(v_2-v_1)\dd{t}\biggr\vert_{0,\alpha;B_1}\\
	&\quad \leq \lVert \D\mathcal{F}^r[0]^{-1}\rVert  \frac{C_4^{-1}}{2}  \leq \frac{1}{2},
	\end{split}
	\end{equation}
	showing that it is indeed a contraction, see \cite[Theorem 3.1.5]{MR0488101} for details.
	It remains to check that $T_r$ is a mapping from the ball centered at $0$ of radius $r^{2+\alpha}\delta/C_6$ into itself. By \cite[Theorem 3.1.5]{MR0488101}, it suffices to check that 
	\begin{equation}\label{inversemappingfinalinequality}
		|\mathcal{F}^r[0]|_{0,\alpha;B_1} < \frac{1}{2}r^{2+\alpha}\delta/C_6  \Vert \D\mathcal{F}^r[0]^{-1}\Vert^{-1} ,
	\end{equation}
	for $r$ sufficiently small, since then
	\begin{equation*}
	\begin{split}
		|T_r[v]|_{2,\alpha;B_1} &\leq |T_r[0]|_{2,\alpha;B_1} + \frac{1}{2}|v|_{2,\alpha;B_1}\\
		& \leq \lVert D \mathcal{F}^r[0]^{-1}\rVert\lvert \mathcal{F}^r[0]\rvert_{0,\alpha;B_1} + \frac{1}{2}r^{2+\alpha}\delta/C_6\\
		& \leq r^{2+\alpha} \delta/C_6,
	\end{split}
	\end{equation*}
where the first inequality follows from \eqref{contractivemapping}.
	 By \eqref{F_r[P_r]decreaserate}, \eqref{DF_rdecreaserate} and \eqref{DF_rdecreaseratelb}, the left hand side and the right hand side of \eqref{inversemappingfinalinequality} are of order of $r^{3}$ and $r^{2+\alpha}$ respectively. Thus \eqref{inversemappingfinalinequality} holds for all $r \leq r_0$, with $(r_0)^{1-\alpha} =C_8(f,P,\alpha,L,R, W)\delta$, proving the existence of a $w^{(r)}$ satisfying \eqref{propertitysofWupr}.\par 
	 \noindent \textbf{Step 5.} It remains to get estimates on the higher order derivatives of $w^{(r)}$ constructed above. Taking the partial derivative of \eqref{propertitysofWupr} with respect to $y^k$, one obtains
	 \begin{equation}\label{FrechletderivativeofFstandingonlocalsolution}
	 	\overline{a}^{ij}_r \partial_{kij} w^{(r)} + \overline{b}^i_r  \partial_{ki}w^{(r)} + \overline{c}_r  \partial_k w^{(r)}= -\overline{d}_r, 
	 \end{equation}
	 where 
	 	 \begin{equation}\label{overlineabcrdef}\left\{
	 	\begin{aligned}
	 		\overline{a}^{ij}_r(y) &= \sum_{sl}\biggl(\pdv{}{M_{sl}} (f\circ \lambda)\biggr)\bigl(\bigl( W_r \circ (\D_y^2 w^{(r)} + L^r(\cdot  ,\D_y w^{(r)}))  \circ W_r \bigr)(y)\bigr)\\ & \quad \cdot  (W_{r})_{li} \cdot (W_{r})_{js}, \\
	 		\overline{b}_r^s(y)&=\sum_{ij}\overline{a}^{ij}_r(y)  (\partial_{p^s}L_{ij}^{r}) (y, \D_y w^{(r)}),  \\
	 		\overline{c}_r(y)&=- 2R_r e^{2w^{(r)}},\\
	 		\overline{d}_r(y)& = \sum_{sl}\biggl(\pdv{}{M_{sl}} (f\circ \lambda)\biggr)\bigg(\bigl( \partial_k W_r \circ (\D_y^2 w^{(r)}+ L^r(\cdot  ,\D_y w^{(r)})) \circ W_r\bigr)_{sl}\\
	 		&\phantom{= \sum_{sl}\biggl(\pdv{}{M_{sl}} (f\circ \lambda)\biggr)\bigg(}+\bigl( W_r \circ ((\partial_k L^r)(\cdot  ,\D_y w^{(r)})) \circ W_r\bigr)_{sl}\\
	 		&\phantom{= \sum_{sl}\biggl(\pdv{}{M_{sl}} (f\circ \lambda)\biggr)\bigg(}	+\bigl(  W_r \circ (\D_y^2 w^{(r)} + L^r(\cdot  ,\D_y w^{(r)})) \circ \partial_k W_r\bigr)_{sl}
	 		\bigg)\\
	 		&-(\partial_k R_r) e^{2w^{(r)}}.
	 	\end{aligned}\right.
	 \end{equation}
 By \eqref{conecondition}, $\pdvs{f}{\lambda_i}$ is homogeneous of degree zero, and due to the conditions \eqref{structurealconditionsoinlnL}, \eqref{blowup} on $L$ and $L^r$ respectively, the $C^{2, \alpha}$ estimate of $w^{(r)} - P_r$ in \eqref{propertitysofWupr} implies that there exist $\overline{\lambda}$ and $\overline{\Lambda}$, independent of $r$, such that 
 \begin{equation}
\begin{aligned}
\overline{\lambda} |\xi|^2 &\leq \overline{a}_r^{ij} \xi_i \xi_j & |\overline{a}^{ij}_r|_{0,\alpha;B_1}+ |\overline{b}^{i}_r|_{0,\alpha;B_1}+|\overline{c}_r|_{0,\alpha;B_1} & \leq \overline{\Lambda}.
\end{aligned}
 \end{equation} for all $r$ small and $\xi \in \mathbb{R}^n$. Additionally, by \eqref{blowup}, 
 \begin{equation}\label{estimatingdrpt1}
	\begin{aligned}
	|W_r|_{0,\alpha} &=O(1), & |\partial_k W_r|_{0,\alpha}&=O(r),\\
	|R_r|_{0,\alpha} &=O(1), & |\partial_k R_r|_{0,\alpha}&= O(r),\\
	|e^{2w^{(r)}}|_{0,\alpha} & = O(r^2).
	\end{aligned}
 \end{equation}
 and by the estimate $|w^{(r)} - P_r|_{2,\alpha;B_1} = O(r^{2 + \alpha})$, which we established in the previous part, 
 \begin{equation}\label{estimatingdrpt2}
 	\begin{aligned}
|\D_y^2 w^{(r)}|_{0,\alpha} \leq|\D_y^2 (w^{(r)} - P_r)|_{0,\alpha} &+ |\D_y^2 P_r|_{0,\alpha} = O(r^{2 + \alpha}) +  O(r^2) = O(r^2),\\
|L^r(y, \D_y w^{(r)})|_{0,\alpha} &= O(r^2),\\
|(\partial_kL^r)(y, \D_y w^{(r)})|_{0,\alpha} &= O(r^3).
 	\end{aligned}
 \end{equation} 
 Combining \eqref{estimatingdrpt1} and \eqref{estimatingdrpt2}, we obtain 
 \begin{equation}
|\overline{d}_r|_{0,\alpha} = O(r^3).
 \end{equation}\par 
  By classical interior Schauder estimate \cite[Theorem 6.2]{MR0473443}, 
 \begin{equation*}
 \begin{split}
 	|(w^{(r)} - P_r)_k|_{2,\alpha;B_{3/4}} &\leq C_{11}(n,\alpha, \overline{\lambda}, \overline{\Lambda}) \bigl(|(w^{(r)} - P_r)_k|_{0;B_1} + |\overline{d}_r|_{0,\alpha} \bigl)\\
 	& \leq C_{12}r^{2 + \alpha}. 
 \end{split}
 \end{equation*}\par 
 Differentiating \eqref{propertitysofWupr} once again and repeating the argument above, we obtain the following estimate on the fourth order derivative of $w^{(r)} - P_r$, 
\begin{equation*}
	\begin{split}
		|\D^2(w^{(r)} - P_r)|_{2,\alpha;B_{1/2}} \leq C_{13}r^{2 + \alpha}, 
	\end{split}
\end{equation*}
completing the proof. 

\end{proof}
We still need one technical lemma.  \par 
For any $u \in C^{2}(B_1)$, consider $F_u \colon \text{Sym}_n \times \mathbb{R}^n \times \mathbb{R} \times B_1 \to \mathbb{R}$ defined by
\begin{equation}\label{Fuform}
	F_u(M,p,z,x)= e^{-2u} f\bigl[\lambda\bigl( W \circ (M+ \D^2_x u+ L(\cdot ,p+\D_x u))\circ W \bigr)\bigr](x)- R(x)e^{2z}.
\end{equation}
Similarly, we define 
\begin{equation}\label{Furform}
	\begin{split}
	F_u^r(M,p,z,y)&= e^{-2u} f\bigl[\lambda\bigl( W_{r} \circ(r^2M+ \D^2_y u+ L^r(\cdot ,r^2p+\D_y u))  \circ W_{r} \bigr)\bigr](y)\\
	&\quad   - R_r(y)e^{2r^2z}.
	\end{split}
\end{equation}
\begin{lemma}\label{blowupFsavinasd}
Let $u \in C^2(B_1)$ and $u_r$ be the blowup of $u$ such that 
\begin{align*}
u_r(y)&= u(x)+\ln r, & x&= ry. 
\end{align*}
If for some $\lambda, \Lambda, \delta$,  $F_u$ satisfies \eqref{SavinP1} in \autoref{Savinisgod}, then $F^r_{u_r}$ satisfies \eqref{SavinP1}  in \autoref{Savinisgod} for the same constants $\lambda,\Lambda,\delta$ in $B_{1}$.
\end{lemma}
\begin{proof}
By \eqref{structurealconditionsoinlnL}, \eqref{blowup}, 
\begin{equation*}
L^r(y,rp)= \sum_{k} r^{2-k}L_k(x,rp)= \sum_kr^2 L_k(x,p)= r^2L(x,p).
\end{equation*}
The claim follows from the following equality,  
\begin{equation}\label{relationshipbetweenFuruandFu}
	\begin{split}
		&F_{u_r}^r(M,p,z,y)\\
		=&e^{-2u_r}  f\bigl[\lambda\bigl( W_r \circ(r^2M+\D_y^2 u_r+ L^r(\cdot,r^2p+\D_y u_r))\circ  W_r \bigr)\bigr](y)   - R_r(y)e^{2r^2z}\\ =&e^{-2u}  f\bigl[\lambda\bigl( W \circ(M+\D_x^2 u+ L(\cdot,rp+\D_x u))\circ  W \bigr)\bigr](x)  - R(x)e^{2r^2z}\\
		=& F_u(M,rp,r^2z,ry),
	\end{split}
\end{equation}
where $ry=x$, completing the proof of \autoref{blowupFsavinasd}.
\end{proof}

\begin{proof}[Proof of \autoref{C2alpha}]
Let $w$ be the viscosity solution in the theorem. Since $w$ is punctually second-order differentiable at $0$, both $\D w(0)$ and $\D^2 w(0)$ are well defined. Let \begin{equation}
	P(x)= w(0)+x \cdot \D w(0) + \frac{1}{2} x^\intercal  (\D^2 w(0))  x,
\end{equation}
and define $P_r$, $W_r$, $L_r$, and $R_r$ by \eqref{blowup} for $0 < r \leq  1$.\par 
 After replacing $P$ by $P_r$, $W$ by $W_r$, $L$ by $L_r$, and $R$ by $R_r$, for some $r$ sufficiently small, we may assume
\begin{equation}
	\lambda\bigl(W \circ (\D^2 P + L(\cdot ,\D P)) \circ W\bigr)(x) \in \Gamma,
\end{equation}
for all $x \in B_1$.\par 
 Pick $\bar{\delta} > 0$ so that $|u-P|_{2,\alpha/2;B_1} < \bar{\delta}$ implies that \eqref{SavinP1} in \autoref{Savinisgod} holds for any $F_{u}$ for some fixed $\lambda, \Lambda, \delta$ in $B_{1}$.\par 

By \autoref{punctuallysecondincone}, we may apply \autoref{inversemappingsigma2} to the paraboloid $P$ and obtain for any $r$ sufficiently small, a function $w^{(r)} \in C^{2,\alpha}(\overline{B}_1)$ satisfying \eqref{propertitysofWupr}. \par 
Let $w_r$ be defined by 
\begin{align}\label{properblowupwu}
	w_r(y)&= w(x)+\ln r, & x&= ry. 
\end{align} By \eqref{Furform}, $r^{-2}(w_r - w^{(r)})$ satisfies 
\begin{equation}\label{wr-w^risasolutiontothenewfunction}
	F^r_{w^{(r)}}\bigl(\D^2\bigl( r^{-2}(w_r - w^{(r)})\bigr), \D \bigl(r^{-2}(w_r - w^{(r)})\bigr), r^{-2}(w_r - w^{(r)}), y\bigr)=0,
\end{equation}
in the viscosity sense. \par  
Since $w^{(r)}$ satisfies \eqref{propertitysofWupr}, it is clear that
\begin{equation}\label{zeroisasolution}
	\begin{multlined}
		F_{w^{(r)}}^r(0,0,0,y)
		=\\ e^{-2w^{(r)}}f\bigl[\lambda\bigl( W_{r} \circ\bigl(\D_y^2 w^{(r)} + L^r(\cdot ,\D_y w^{(r)})\bigr) \circ W_{r}  \bigr)\bigr](y) -  R_r 
		=0,
	\end{multlined}
\end{equation}
hence \eqref{SavinP2.5} in \autoref{Savinisgod} is satisfied.\par 

 By \cite[Lemma 6.37]{MR0473443} and \eqref{propertitysofWupr}, there is $v \in C^{2,\alpha/2}(\overline{B}_1) \cap C^{4, \alpha} (B_{r/2})$ with 
\begin{equation}\label{higherorderestimatesoftheblowup}\left\{
	\begin{aligned}
	|v-P|_{2,\alpha/2;B_1} &\leq 2C_7 r^{-2 -\alpha/2} \vert w^{(r)}-P_r\vert_{2,\alpha;B_{1}} \leq 2r^{\alpha/2} C_2  C_7\\
	|\D^3(v - P)|_{0,\alpha/2; B_{r/2}} & \leq 2C_7 r^{-3 - \alpha/2}\vert w^{(r)}-P_r\vert_{3,\alpha;B_{1/2}} \leq 2r^{\alpha/2-1} C_2  C_7,\\
	|\D^4(v - P)|_{0,\alpha/2; B_{r/2}} & \leq 2C_7 r^{-4 - \alpha/2}\vert w^{(r)}-P_r\vert_{4,\alpha;B_{1/2}} \leq 2r^{\alpha/2-2} C_2  C_7,\\
	w^{(r)}(y)&=v(ry)+\ln r,
\end{aligned} \right.
\end{equation}	
For the construction of $v$, one can also refer to \textbf{Step 3} in the proof of \autoref{inversemappingsigma2}.
By the choice of $\bar{\delta}$ and \autoref{blowupFsavinasd}, $F^r_{w^{(r)}}(M,p,z,y)$
	satisfies \eqref{SavinP1} in \autoref{Savinisgod} for the same $\lambda, \Lambda, \delta$ in $B_{1}$ for all $r$ satisfying \begin{equation*}
		2r^{\alpha/2}  C_2 C_7 \leq \bar{\delta}.
		\end{equation*}  Additionally, by \eqref{relationshipbetweenFuruandFu}, 
		\begin{equation}\label{relationshipbetweenFwrwandFw}
		F_{w^{(r)}}^r(M, p, z, y) = F_{v}(M, rp, r^2z, ry). 
		\end{equation}
		It is then clear from \eqref{higherorderestimatesoftheblowup}, \eqref{relationshipbetweenFwrwandFw}, and the expression of $F_v$, \eqref{Fuform}, that $|\D^2 F_{w^{(r)}}^r|_{0,B_{1/2}}$ has an upper bound independent of $r$. In other words, $F^r_{w^{(r)}}(M,p,z,y)$
		satisfies \eqref{SavinP2} in \autoref{Savinisgod} for the same constant $K$ in $B_{1/2}$ for all $r$ small. \par 
	It remains to get an estimate on $\vert r^{-2}(w_r - w^{(r)})\vert_{0;B_{1/2}}$. By \eqref{properblowupwu} and \eqref{propertitysofWupr}, we have 
	\begin{equation}\label{risasolutiontothenewfunctiodecrease}
		\begin{split}
			\vert r^{-2}( w_r-w^{(r)})\vert_{0;B_{1/2}} &\leq \vert r^{-2}( w_r-P_r)\vert_{0;B_{1/2}} +\vert r^{-2}( P_r-w^{(r)})\vert_{0;B_{1/2}} \\
			&\leq r^{-2}o(r^{2})+r^\alpha  C(P,\alpha),
		\end{split}
	\end{equation} 
	which converges to $0$ as $r \to 0$.\par Since $F_{w^{(r)}}^r$ satisfies \eqref{SavinP1} and \eqref{SavinP2} on $B_{1/2}$ for coefficients independent of $r$, and we have \eqref{zeroisasolution}, \eqref{wr-w^risasolutiontothenewfunction}, and \eqref{risasolutiontothenewfunctiodecrease}, it follows that \autoref{Savinisgod} applies to the pair $(F_{w^{(r)}}^r,r^{-2}(w_r-w^{(r)}))$  on $B_{1/2}$ for $r$ sufficiently small. Therefore $r^{-2}(w_r - w^{(r)})$, hence $w$ is $C^{2,\alpha}$ in a neighborhood of $0$, completing the proof of \autoref{C2alpha}.
\end{proof}
Combining \autoref{C2alpha} and \autoref{2convexity}, we immediately conclude \autoref{maintheorempdeside}. 
\begin{corollary}\label{maintheorempdeside}
	Let $\Gamma \subset \Gamma_k^+$ and $w \in C^{0,1}(B_1)$ be a viscosity solution to \eqref{moregeneralL}. If either 
	\begin{enumerate}
		\item  $k > n/2$ and $W(x)= \delta_{ij}$ for all $x \in B_1$, or 
			\item $k=n$,
	\end{enumerate}
then there is a closed set $S \subset B_1$ and $S$ is of measure $0$ such that $w\in C^{2,\alpha}_{\text{loc}}(B_1 - S)$.
\end{corollary}
Next, we give a proof of \autoref{corollary2}.
\begin{proof}[Proof of \autoref{corollary2}]
By \cite{MR1976082}, for any $\tau \in [0,1)$, there exists $w_\tau \in C^\infty(M)$ solving \eqref{modifiedschouten}. Moreover, $\sup_{\tau \in [0,1)} |w_\tau|_{0,1;M}< \infty$. By \cite[Appendix]{MR4210287}, $w_\tau$ converges to a Lipschitz viscosity solution $w$ to \eqref{sigmakproblem}. By \autoref{punctuallydifferentiablealmosteverysmoothapproximation}, $w$ is punctually second-order differentiable almost everywhere on $M$. Applying \autoref{C2alpha} to $w$ at points where $w$ is punctually second-order differentiable, we obtain the desired result.
\end{proof}
\section{Equations of Krylov Type}
\label{Krylovtypesection}
The method in \autoref{localregularitysection} also applies to \eqref{Krylovstandard}. \par 
Let $(B_1,g)$ be a Riemannian manifold. For $0 <\epsilon < \alpha_l \in C^\infty(B_1)$ with $0 \leq l \leq k-2$ and $\alpha \in C^\infty(B_1)$, define 
\begin{equation*}
	\widetilde{f}(\lambda,z,x) = \frac{\sigma_k}{\sigma_{k-1}} \bigl(\lambda\bigr) - \sum_{l=0}^{k-2} \alpha_l(x) e^{(2 (k-l) z)} \frac{\sigma_l}{\sigma_{k-1}} \bigl( \lambda),
\end{equation*}
then \eqref{Krylovstandard} takes the form
\begin{equation}\label{PDEformofKrylov}\left\{
	\begin{aligned}
\widetilde{f}\bigl[\lambda(-g^{-1}A_{g_w}),w,\cdot\bigr] (x) &=-\alpha(x) e^{2w},\\
\lambda(-g^{-1}A_{g_w}) \in \Gamma_{k-1}^+,
	\end{aligned}\right.
\end{equation}
where $g_w^{-1}A_{g_w}= e^{-2w}  g^{-1}A_{g_w}$.
\eqref{PDEformofKrylov} is also a fully nonlinear non-uniformly elliptic PDE. 
\begin{theorem}\cites{MR4767573,MR4278951}\label{elliptickrylov}
Let $w \in C^\infty(B_1)$. If $\lambda(-g^{-1}A_{g_w}) \in \Gamma_{k-1}^+$ and $\alpha_l \geq 0$ for $0 \leq l \leq k-2$, then \eqref{PDEformofKrylov} is elliptic and concave.
\end{theorem} 

In view of \autoref{elliptickrylov}, we can define the viscosity solution to \eqref{Krylovstandard} similarly to \autoref{defnviscositysolution}.

 \begin{definition}\label{newKrylovstanderfpde}
 	$w \in C(B_1)$ is a viscosity subsolution to \eqref{Krylovstandard} if for all $x_0 \in B_1$, for all $\varphi \in C^2$ and for all neighborhoods $U$ of $x_0$ with 
 	\begin{align*}
 		\varphi(x_0)&= w(x_0) & \varphi \geq w \text{ in }U,
 	\end{align*}
 	there holds, 
 	\begin{equation*}
 		\begin{aligned}
 			\lambda\bigl(-g^{-1}A_{g_\varphi}\bigr)(x_0) &\in \Gamma_{k-1}^+  \quad \text{ and } \\ \widetilde{f}\bigl[\lambda\bigl(-g^{-1}A_{g_\varphi}\bigr),\varphi, \cdot \bigr](x_0) &\geq -\alpha(x_0)e^{2 w(x_0)}.
 		\end{aligned}
 	\end{equation*} \par 
 	$w \in C(B_1)$ is a viscosity supersolution to \eqref{Krylovstandard} if for all $x_0 \in B_1$, for all $\varphi\in C^2$ and for all neighborhoods $U$ of $x_0$ with
 	\begin{align*}
 		\varphi(x_0)&= w(x_0) & \varphi \leq w \text{ in }U,
 	\end{align*}
 	either 
 	\begin{equation*}
 		\begin{aligned}
 			\lambda\bigl(-g^{-1}A_{g_\varphi}\bigr)(x_0) &\in \Gamma_{k-1}^+  \quad \text{ and } \\ \widetilde{f}\bigl[\lambda\bigl(-g^{-1}A_{g_\varphi}\bigr),\varphi, \cdot \bigr](x_0) &\leq -\alpha(x_0)e^{2 w(x_0)},
 		\end{aligned}
 	\end{equation*}
 	or 
 	\begin{equation*}
 		\lambda\bigl(-g^{-1}A_{g_\varphi}\bigr)(x_0) \in \mathbb{R}^n-\Gamma_{k-1}^+.
 	\end{equation*}\par 
 	$w \in C(B_1)$ is a viscosity solution to \eqref{Krylovstandard} if it is both a viscosity subsolution and a viscosity supersolution to \eqref{Krylovstandard}.
 \end{definition} 
We observe the following result \autoref{Krylovexistence}, which essentially follows from \cite{MR4767573} using an argument in \cite{MR4210287}.
\begin{theorem}\cite{MR4767573}\label{Krylovexistence}
Notations as above. Let $(M,g)$ be a closed Riemannian manifold with $\lambda(-g^{-1}A_{g}) \in \Gamma_{k}^+$. If $\alpha, \alpha_l \in C^\infty(M)$ and $\alpha_l  > 0$ for all $l=0,\cdots, k-2$, then there exists a Lipschitz viscosity solution $w$ to \eqref{PDEformofKrylov} such that $w$ is also a viscosity subsolution to 
\begin{equation}\label{strictlyinsigmakcone}
	\left\{
	\begin{aligned}
		\sigma_{k-1}(-g^{-1}_wA_{g_w}) &= C, \\ \lambda(-g^{-1}_wA_{g_w}) &\in \Gamma_{k-1}^+,	\end{aligned} \right.
\end{equation} for some $C>0$ depending on $\alpha_l$, $\alpha$, $k$, and $(M,g)$. \par 
Moreover, if $k-1>n/2$, then $w$ is punctually second-order differentiable almost everywhere on $M$.
\end{theorem}
\begin{proof}
In \cite{MR4767573}, {Chen and He} consider the following PDE
\begin{equation}\label{PDEkrylovstau} \left\{
	\begin{aligned}
		A_{g_w}^\tau&= \frac{1}{n-2} \biggl(\text{Ric}_{g_w} - \frac{\tau \cdot \text{Sc}_{g_w}}{2(n-1)}g_w \biggr),\\
		\widetilde{f}\bigl[\lambda(-g^{-1}A_{g_w}^\tau),w,\cdot\bigr] (x) &=-\alpha(x) e^{2w},\\
	\end{aligned}\right. 
\end{equation}
for $0\leq \tau < 1$. They show that when
\begin{equation}\label{AgtauliesinKcone}
	\lambda(-g^{-1}A_{g}^\tau) \in \Gamma_{k}^+,
\end{equation}  and $\alpha_l > 0$ for all $0 \leq l \leq k-2$, there exists $w_\tau \in C^\infty(M)$ such that $w_\tau$ solves \eqref{PDEkrylovstau}, and $|w_\tau|_{C^1(M)}$ is bounded by some constant depending only on $\alpha_l$, $\alpha$, $k$ and $(M,g)$. Moreover, $w_\tau$ satisfies 
\begin{equation}\label{strictlowerboundHeChen}
	\lambda(-g_{w_{\tau}}^{-1}A_{g_{w_{\tau}}}^\tau) \in \{\lambda \in \Gamma_{k-1}^+; \sigma_{k-1}(\lambda)>C\},
\end{equation}
for some $C>0$ depending on $\alpha_l$, $\alpha$, $k$, and $(M,g)$.\par 
For $0\leq \tau < 1$, the assumption $\lambda(-g^{-1}A_{g})\in \Gamma_k^+$ implies that 
\begin{equation*}
	\lambda(-g^{-1}A_{g}^\tau) = \lambda(-g^{-1}A_{g})-\frac{1-\tau}{2(n-1)(n - 2)} \text{Sc}_g \cdot (1,1,\cdots, 1) \in \Gamma_k^+.
\end{equation*}
Hence \eqref{AgtauliesinKcone} is satisfied. \par 
Along some subsequence $\tau_i \to 1$, $w_{\tau_i}$ converges uniformly to a Lipschitz function $w$. By the same argument as in \cite[Appendix]{MR4210287}, $w$ is a Lipschitz viscosity solution to \eqref{PDEformofKrylov} and is a viscosity subsolution to \eqref{strictlyinsigmakcone}. In case $k-1>n/2$, \autoref{punctuallydifferentiablealmosteverysmoothapproximation} applies to $w$, showing that $w$ is punctually second-order differentiable almost everywhere. 
\end{proof}
Due to the proof of \autoref{Krylovexistence}, \autoref{notinterestingtechnicallemma} and \autoref{checkconditionnaturalgood} also apply to the viscosity solution in \autoref{Krylovexistence}.

\autoref{punctuallysecondincone} also holds for the viscosity solution $w$ in \autoref{Krylovexistence}. 
\begin{lemma}
Let $\alpha, \alpha_l \in C^\infty(B_1)$ and $\alpha_l > \epsilon > 0$ for all $l=0,\cdots, k-2$.
Let $w$ be a viscosity solution to \eqref{PDEformofKrylov} that is also a viscosity subsolution to \eqref{strictlyinsigmakcone} for some $C>0$. \par 
 If $w$ is punctually second-order differentiable at $x_0$, then 
\begin{equation*}\left\{
\begin{aligned}
\widetilde{f}(\lambda(-g^{-1}A_{g_w}), w, \cdot )(x_0) &=- \alpha(x_0) e^{2w(x_0)},\\
\lambda(-g_w^{-1}A_{g_w})(x_0) &\in \Gamma_{k-1}^+.
\end{aligned}	
	\right.
\end{equation*}
\end{lemma}
\begin{proof}
Let $P$ be the paraboloid such that 
\begin{equation}\label{punctuallysecondorderwde2}
	w(x)=P(x)+o(|x-x_0|^2). 
\end{equation}
Then by \eqref{punctuallysecondorderwde2}, for $\epsilon>0$, 
\begin{align*}
	P(x_0)&= w(x_0) & P(x) + \epsilon |x-x_0|^2 &\geq w(x) \text{ in a neighborhood of }x_0.
\end{align*}
Since $w$ satisfies \eqref{strictlyinsigmakcone} in the viscosity sense, 
\begin{equation*}\left\{
\begin{aligned}
	\lambda(-g_{P+\epsilon |x-x_0|^2}^{-1}A_{g_{P+\epsilon |x-x_0|^2}})(x_0) &\in\Gamma_{k-1}^+ ,\\ \sigma_{k-1}(\lambda(-g_{P+\epsilon |x-x_0|^2}^{-1}A_{g_{P+\epsilon |x-x_0|^2}}))(x_0) \geq C &> 0.
\end{aligned}\right.
\end{equation*}
Pushing $\epsilon \to 0$, we conclude that 
\begin{equation*}
	\lambda(-g_{P}^{-1}A_{g_{P}})(x_0) \in\Gamma_{k-1}^+.
\end{equation*}
The remaining part of the proof is identical to that of \autoref{punctuallysecondincone}.
\end{proof}
 Likewise, \autoref{inversemappingsigma2} also holds for viscosity solutions to \eqref{PDEformofKrylov} with minor changes in the setup of the proof, which we state below. In the following, we consider \eqref{PDEformofKrylov} in local coordinates. Take $W_{ij}(x)=\sqrt{g^{ij}(x)}$ and define $L$ by \eqref{Lforschouten}.
\begin{proposition}\label{variantinversemappingsigma2}
Notations as above. Let $\alpha, \alpha_l \in C^\infty(B_1)$ and $\alpha_l > \epsilon > 0$ for all $l=0,\cdots, k-2$. Let $P$ be a paraboloid such that 
	\begin{equation}\label{inversemappingstartingassumptionforgeneral}\left\{
	\begin{aligned}
		\widetilde{f}\bigl[\lambda\bigl( W \circ (\D^2 P + L(\cdot ,\D P) )\circ W \bigr)\bigr](0)=&-\alpha(0) e^{2P(0)}, \\
		\lambda\bigl( W \circ (\D^2 P + L(\cdot ,\D P) ) \circ W \bigr)(0) \in & \Gamma_{k-1}^+.
	\end{aligned}\right.
\end{equation}
Then there exist constants $C_9, C_{10}$, depending on $k, P,\gamma,\alpha, \alpha_l$, and $W$, such that for all $r^{1-\gamma} \leq C_9$, there is $w^{(r)} \in C^{2,\gamma}(\overline{B}_1)$ satisfying 
\begin{equation*}\left\{
	\begin{alignedat}{2}
			\widetilde{f}^r\bigl[\lambda\bigl( W_r \circ (\D^2_y w^{(r)} + L^r(\cdot  ,\D_y w^{(r)}))  \circ W_r \bigr),w^{(r)},\cdot \bigr](y) &=  -\alpha_r e^{2w^{(r)}} \quad  &&\text{ in } B_1,\\
		\lambda\bigl( W_r \circ (\D^2 w^{(r)} + L^r(\cdot  ,\D_y w^{(r)}))  \circ W_r \bigr)(y)&\in \Gamma_{k-1}^{+} && \text{ in }B_1,\\
		|w^{(r)} - P_r|_{2,\gamma;B_1}, |\D_y^3 (w^{(r)} - P_r)|_{0,\gamma;B_{1/2}}, |\D_y^4 (w^{(r)} - P_r)|_{0,\gamma;B_{1/2}}   &\leq r^{2 + \gamma}C_{10},\\
		w^{(r)}&= P_r \quad && \text{ on } \partial B_1,	
	\end{alignedat}\right.
\end{equation*}
where 
\begin{equation}\label{blowupnew}\left\{
	\begin{aligned}
		L^r(y,p) &= r^2L_0(ry,p)+ rL_1(ry,p)+ L_2(ry,p),\\
		P_r(y)&=P(ry)+\ln r, \\
		\alpha_{l,r}(y) &= \alpha_l(ry),\\
		 \alpha_r(y) &= \alpha(ry),\\
				\widetilde{f}^r(\lambda, z,y)&= \frac{\sigma_k}{\sigma_{k-1}} \bigl(\lambda\bigr) -\sum_{l=0}^{k - 2} \alpha_{l,r}(y) e^{(2 (k-l) z)} \frac{\sigma_l}{\sigma_{k-1}} \bigl( \lambda),\\
			 W_{r}(y) &= W(ry),
	\end{aligned}\right.
\end{equation}
for $y \in B_1$ and $0< r \leq 1$.
\end{proposition}
\begin{proof}
Let $P$ be the paraboloid in the proposition. Define $P_r$, $\alpha_{l,r}$, $\alpha_r$, $W_r$, $\widetilde{f}^r$ and $L^r$ by \eqref{blowupnew}. We also adopt the coordinate transformation defined by \eqref{coordinatetransformation}.

Define $\widetilde{\mathcal{F}}^r \colon  C^{2,\gamma}(\overline{B}_1) \cap C_0 (\overline{B}_1) \to C^\gamma (\overline{B}_1)$ by
\begin{equation}\label{newFr}
\begin{aligned}
	\widetilde{\mathcal{F}}^r[u] &=
	\widetilde{f}^r\bigl[\lambda\bigl( W_{r}(y) \circ (\D_y^2 (u+P_r) + L^r(\cdot  ,\D_y (u+P_r)) ) \circ W_{r}(y)\bigr),\\
	& \qquad u+P_r,y\bigr]  + \alpha_re^{2(u+P_r)}.
\end{aligned}
\end{equation}
We need only to check that \eqref{scalingforFr} and \eqref{abcblowup} still hold for $\widetilde{\mathcal{F}}^r$ newly defined by \eqref{newFr}. The remaining steps are identical. \par 
We have 
\begin{equation}\label{newscalingFr}
	\begin{split}
\widetilde{\mathcal{F}}^r[0](y)&= \widetilde{f}^r\bigl[\lambda\bigl( W_{r}(y)\circ (\D_y^2 P_r + L^r(\cdot  ,\D_y P_r)) \circ W_{r}(y)\bigr), P_r, y\bigr] +\alpha_r e^{2(P_r)}\\
&= r^2 \widetilde{f}\bigl[\lambda\bigl( W_{r}(y)\circ (\D_x^2 P + L(\cdot  ,\D_x P)) \circ W_{r}(y)\bigr), P, x\bigr] + r^2\alpha(x) e^{2(P)}\\
&= r^2 \widetilde{\mathcal{F}}[0](x).
	\end{split}
\end{equation}\par 
Now let $a_r^{ij}, b^i_r,c_r$ be the coefficients of $\D \widetilde{\mathcal{F}}^r[0]$, so that 
\begin{equation*}
\D \widetilde{\mathcal{F}}^r[0](h)= a_r^{ij} h_{ij} +b^i_rh_i+c_rh.
\end{equation*}
Then 
\begin{equation}\label{newarbrcr}\left\{
\begin{aligned}
a_r^{ij}(y)&= \sum_{st}\biggl(\pdv{(\sigma_k/\sigma_{k-1}) \circ \lambda}{M_{st}} -\sum_{l=0}^{k-2} \alpha_{l,r}  e^{2(k-l)P_r }\pdv{(\sigma_l/\sigma_{k-1}) \circ \lambda}{M_{st}}\biggr)\\
	& \quad  (W_{r})_{ti}  (W_{r})_{js},\\
b_r^l(y)&= \sum_{ij} a_r^{ij}  (\partial_{p^l}L_{ij}^r)(y,\D_y P_r),\\
c_r(y)&=- \sum_{l=0}^{k-2}2\alpha_{l,r}(y) (k-l) e^{2(k-l)P_r} \frac{\sigma_l}{\sigma_{k-1}} + 2\alpha_r(y)e^{2P_r}.
\end{aligned}\right.
\end{equation}
Since $\sigma_{k}/\sigma_{k-1}$ is homogeneous of degree $1$ and $\sigma_l/ \sigma_{k-1}$ is homogeneous of degree $l-k+1$, we deduce from \eqref{newarbrcr} that 
	\begin{align}\label{newabcblowup}
	a^{ij}_r(y)&=  a^{ij}(x), &
	b^{i}_r(y)&= r b^{i}(x), &
	c_r(y)&= r^{2} c(x).
\end{align} 
With \eqref{newscalingFr} and \eqref{newabcblowup}, the proof is complete.
\end{proof}
Again, let $L$ be defined by \eqref{Lforschouten}, and $L^r$, $W_r$, $\alpha_r$, and $\widetilde{f}^r$ be defined by \eqref{blowupnew}.
For any $u \in C^{2}(B_1)$, consider $\widetilde{F}_u \colon \text{Sym}_n \times \mathbb{R}^n \times \mathbb{R} \times B_1 \to \mathbb{R}$ defined by
\begin{equation}\label{Fufornewm}
	\begin{split}
	\widetilde{F}_u(M,p,z,x)&= e^{-2u}\cdot \widetilde{f}\bigl[\lambda\bigl( W(x)\circ(M+ \D^2_x u+ L(x ,p+\D_x u)) \circ W(x)\bigr),\\
	& \qquad u+z,\cdot \bigr](x)+ \alpha(x)e^{2z}.
	\end{split}
\end{equation}
Similarly, we define 
\begin{equation}\label{Furfornewm}
	\begin{split}
		\widetilde{F}_u^r(M,p,z,y)&= e^{-2u}\cdot \widetilde{f}^r\bigl[\lambda\bigl( W_r(y)\circ(r^2M+ \D^2_y u+ L^r(y ,r^2p+\D_y u))\\
		&\qquad \circ W_r(y)\bigr),u+r^2 z,\cdot \bigr](y)+ \alpha_r(y)e^{2r^2z}.
	\end{split}
\end{equation}\par 
Following the same steps of calculation, \autoref{blowupFsavinasd} holds for \eqref{Fufornewm} and \eqref{Furfornewm} in place of \eqref{Fuform} and \eqref{Furform}. The same proof of \autoref{C2alpha} works for viscosity solutions to \eqref{newKrylovstanderfpde} with no modifications necessary. We arrive at \autoref{C2alphaforkrylov}.
\begin{theorem}\label{C2alphaforkrylov}
	Let $\alpha, \alpha_l \in C^\infty(B_1)$ and $\alpha_l > \epsilon > 0$ for all $l=0,\cdots, k-2$.
	Let $w$ be a viscosity solution to \eqref{PDEformofKrylov} in $B_1$ that is also a viscosity subsolution to \eqref{strictlyinsigmakcone} for some $C>0$. If $w$ is punctually second-order differentiable at $0$,
	then $w$ is smooth in a neighborhood of $0$.
\end{theorem}
Combining \autoref{C2alphaforkrylov} and \autoref{2convexity}, we immediately conclude \autoref{maintheorempdesidekrylov}. 
\begin{corollary}\label{maintheorempdesidekrylov}
	Let $(B_1,g)$ be a Euclidean domain with the canonical Euclidean metric.
	Let $\alpha, \alpha_l \in C^\infty(B_1)$ and $\alpha_l > \epsilon > 0$ for all $l=0,\cdots, k-2$.
	 Let $w$ be a viscosity solution to \eqref{PDEformofKrylov} on $(B_1,g)$ and a viscosity subsolution to \eqref{strictlyinsigmakcone} for some $C>0$. If $w$ is locally Lipschitz and $k-1>n/2$, then there is a closed  zero measure set $S \subset B_1$ such that $w\in C^{\infty}(B_1 \setminus S)$.
\end{corollary}
\section{Data availability statement}
No data were generated or analyzed during this study.

\printbibliography


\end{document}